\documentclass[11pt,twoside]{article}
\usepackage{latexsym}
\usepackage{amssymb,amsbsy,amsmath,amsfonts,amssymb,amscd}
\usepackage{amsthm}
\usepackage{epsfig, graphicx,subfigure}
\usepackage[active]{srcltx}

\usepackage{colortbl}
\newcommand{\commentout}[1]{}
\newcommand{\R}{\mathbb{R}}
\newcommand{\N}{\mathbb{N}}

\newcommand {\al} {\alpha}
\newcommand {\e}  {\varepsilon}

\newcommand {\Chi} {{\bf \raise 2pt \hbox{$\chi$}} }

\newcommand {\F} { {\mathcal F} }

\newcommand {\U} { {\mathcal U} }

\newcommand {\M} { {\mathcal M} }
\newcommand {\f}   {\frac}
\newcommand {\p}   {\partial}

\newcommand{\dis}{\displaystyle}
\newcommand{\beq}{\begin{equation}}
\newcommand{\beqa} {\begin{array}{rl}}
\newcommand{\eeq}{\end{equation}}
\newcommand{\eeqa}{\end{array}}
\newtheorem{theorem}{Theorem}[section]
\newtheorem{lemma}[theorem]{Lemma}

\newtheorem{remark}[theorem]{Remark}
\newtheorem{proposition}[theorem]{Proposition}
\newtheorem{corollary}[theorem]{Corollary}
\newtheorem*{theoremCCM}{Theorem \cite{CCM}}

\title{\LARGE Long-time asymptotics for nonlinear growth-fragmentation equations}

\author{P. Gabriel \thanks{Universit\'e Pierre et Marie Curie-Paris 6, UMR 7598 LJLL, BC187, 4, place de Jussieu, F-75252 Paris cedex 5, France. Email: gabriel@ann.jussieu.fr}}

\date{\today}

\begin{document}
\maketitle
\pagestyle{plain}
\pagenumbering{arabic}

\begin{abstract}
We are interested in the long-time asymptotic behavior of growth-fragmentation equations with a nonlinear growth term.
We present examples for which we can prove either the convergence to a steady state or conversely the existence of periodic solutions.
Using the General Relative Entropy method applied to well chosen self-similar solutions, we show that the equation can ``asymptotically'' be reduced to a system of ODEs.
Then stability results are proved by using a Lyapunov functional, and the existence of periodic solutions is proved with the Poincar\'e-Bendixon theorem or by Hopf bifurcation.
\end{abstract}


\

\noindent{\bf AMS Class. No.} 35B10, 35B32, 35B35, 35B40, 35B42, 35Q92, 37G15, 45K05, 92D25

\

\section{Introduction}

We are interested in growth models which take the form of a {\it mass preserving} fragmentation equation complemented with a transport term.
Such models are used to describe the evolution of a population in which each individual grows and splits or divides.
The individuals can be for instance cells \cite{Bekkal1, Bekkal2, GW2, MetzDiekmann} or polymers \cite{biben, destaing, GabrielTine} and are
structured by a variable $x>0$ which may be size \cite{DMZ,DPZ}, label \cite{Banks2}, protein content \cite{Doumic, Magal}, proliferating parasites content \cite{Bansaye2}, etc.
More precisely, we denote by $u = u(t,x) \geq 0$ the density of individuals of structured variable $x$ at time $t,$
and we consider that the time dynamics of the population are given by the following equation:
\beq\label{eq:frag-drift}\left\{\begin{array}{l}
\f{\p}{\p t}u(t,x) + \f{\p}{\p x} (\tau(t,x) u(t,x))+\mu(t,x)u(t,x) = (\F u)(t,x),\qquad t\geq0,\ x>0,
\vspace{.2cm}\\
u(t,x=0)=0,\qquad t\geq0,
\vspace{.1cm}\\
u(t=0,x)=u_0(x)\geq0,\qquad x>0,
\end{array}\right.\eeq
where $\F$ is a {\it mass conservative} fragmentation operator
\beq\label{eq:fragop}(\F u)(t,x)=\int_x^\infty b(t,y,x)u(t,y)\,dy-\beta(t,x)u(t,x).\eeq
The {\it mass conservation} for the fragmentation operator requires the relation
\beq\label{as:b}\beta(t,x)=\int_0^x\f yx\, b(t,x,y)\,dy.\eeq
The coefficient $\beta(t,y)\geq0$ represents the rate of splitting for a particule of size $y$ at time $t$ and $b(t,y,x)\geq0$ represents the formation rate of a particule of size $x\leq y$ by the fragmentation.
The velocity $\tau(t,x)>0$ in the transport term represents the growth rate of each individual, and $\mu(t,x)\geq0$ is a degradation or death term.

\

We consider that the time dependency of $\tau$ and $\mu$ is of the form
\beq\label{as:timedependency}\tau(t,x)=V(t)\tau(x)\qquad\text{and}\qquad\mu(t,x)=R(t)\mu(x),\eeq
and moreover that the size dependency is a powerlaw:
\beq\label{as:powerlaw}\tau(x)=\tau x^\nu\qquad\text{and}\qquad\mu(x)\equiv\mu.\eeq
The choice of the coefficients $V(t)$ and $R(t)$ depends on the cases we want to analyse.
We give below four examples in which they are nonlinear terms or periodic controls.
The fragmentation coefficients are assumed to be time-independent and to have a \emph{self-similar} structure
\beq\label{as:selfsimfrag}\beta(t,x)=\beta x^\gamma\qquad\text{and}\qquad b(t,x,y)=\f{\beta(x)}{x}\,\kappa\left(\f yx\right),\eeq
where $\kappa$ a nonnegative measure on $[0,1].$
In the sequel we denote by ${\mathcal F}_\gamma$ the fragmentation operator associated to coefficients satisfying these conditions.
Additionally, to obtain convergence results for nonlinear problems, we will sometimes assume that $\kappa$ is a functional kernel, bounded above and below:
\beq\label{as:boundfragker}\exists\,\underline\kappa,\overline\kappa>0,\quad\forall\,z\in[0,1],\quad \underline\kappa\leq\kappa(z)\leq\overline\kappa.\eeq
Notice that for $b$ as in Assumption~\eqref{as:selfsimfrag}, the quantity
$$n_0:=\int_0^1 \kappa(z)\,dz$$
represents the mean number of fragments produced by the fragmentation of an individual.
Moreover relation~\eqref{as:b} becomes $\int_0^1 z\,\kappa(z)\,dz=1,$ which enforces $n_0>1.$
If $\kappa$ is symmetric ($\kappa(z)=\kappa(1-z)$), we even necessarily have $n_0=2.$

\

Now we state our main results concerning different choices for $V(t)$ and $R(t).$
First we investigate the nonlinear growth-fragmentation equations corresponding to the case when $V$ and/or $R$ are functions of the solution $u(t,x)$ itself.
We also consider a model of \emph{polymerization} in which the transport term depends on $u$ and on a solution to an ODE coupled to the growth-fragmentation equation.
The long-time behavior of these equations is investigated under the assumption that $\tau(x)$ is linear (\emph{i.e.} $\nu=1)$ and $\beta$ increasing (\emph{i.e.} $\gamma>0).$
We finish with a study of the long-time asymptotics in the case when $V$ and $R$ are known periodic controls.

\

\paragraph{\bf Example 1. Nonlinear drift-term.}
We consider that the death rate is time independent ($R\equiv1$) and that the transport term depends on the solution itself through the nonlinearity
\beq\label{eq:nonlinear}\f{\p}{\p t} u(t,x) = - f\left(\int x^p u(t,x)\,dx\right) \f{\p}{\p x} \bigl(x\,u(t,x)\bigr) - \mu\, u(t,x) + {\mathcal F}_\gamma u(t,x),\eeq
where $f:\R_+\to\R_+$ is a continuous function which represents the influence of the weighted total population $\int x^pu(t,x)\,dx\ (p>0)$ on the growth process.
Such weak nonlinearities are common in structured populations (see \cite{Hoppensteadt,Iannelli,Webb} for instance).
The stability of the steady states for related models has already been investigated (see \cite{DiekGyllMetz, farkas, Farkas2, Farkas3, MPR}),
but never for the growth-fragmentation equation with the nonlinearities considered here.
We prove in Section~\ref{sec:nonlineardrift} convergence and nonlinear stability results for Equation~\eqref{eq:nonlinear}
in the functional space ${\mathcal H}:=L^2((x+x^r)dx)$ for $r$ large enough, and more precisely in its positive cone denoted by ${\mathcal H}^+.$
These results are stated in the two following theorems. They require that $f$ is continuous and satisfies
\beq\label{as:f}{\mathcal N}:=\left\{I;\ f(I)=\mu\right\}\ \text{is a finite set}\qquad and\qquad \limsup_{I\to\infty}f(I)<\mu.\eeq

\medskip

\begin{theorem}[Convergence]\label{th:nonlindrift:convergence}
Assume that $f$ satisfies Assumption~\eqref{as:f}, that $\gamma\in(0,2],$ and that the fragmentation kernel $\kappa$ satisfies Assumption~\eqref{as:boundfragker}.
Then the number of positive steady states for Equation~\eqref{eq:nonlinear} is equal to $\sharp{\mathcal N}$
and any solution with an initial distribution $u_0\in{\mathcal H}^+$ either converges to one of these steady states or vanishes.
\end{theorem}

\medskip

\begin{theorem}[Local stability]\label{th:stability}
Assume that $f\in{\mathcal C}^1(\R_+)$ satisfies Assumption~\eqref{as:f}, that $\gamma\in(0,2],$ and that the fragmentation kernel $\kappa$ satisfies Assumption~\eqref{as:boundfragker}.
Then the trivial steady state is locally exponentially stable if $f(0)<\mu$ and $p\geq1,$ and unstable if $f(0)>\mu.$
Any nontrivial steady state $u_\infty$ is locally asymptotically stable if $f'\left(\int x^pu_\infty(x)\,dx\right)<0,$
locally exponentially stable if additionally $\kappa\equiv2,$ and unstable if $f'\left(\int x^pu_\infty(x)\,dx\right)>0.$
\end{theorem}

\medskip

A positive steady state $u_\infty$ satifies $f(I_\infty)=\mu,$ where $I_\infty=\int_0^\infty x^pu_\infty(x)\,dx,$ and its local stability depends on the sign of $f(I)-\mu$ around $I_\infty.$
Indeed, if we start with a initial distribution $u_0$ close to $u_\infty$ and if we freeze the growth term $f\bigl(\int_0^\infty x^pu_0(x)\,dx\bigr)$ in Equation~\eqref{eq:nonlinear},
we obtain a linear growth-fragmentation equation with a principal eigenvalue $\Lambda_0=f\bigl(\int_0^\infty x^pu_0(x)\,dx\bigr)-\mu$ (see Section~\ref{ssec:self-sim}).
Thus if $f'(I_\infty)<0$ for instance, the eigenvalue $\Lambda_0$ is positive for an initial data with a $p\,$-moment less than $I_\infty,$
and negative for an initial $p\,$-moment greater than $I_\infty.$
So $u_\infty$ is expected to be stable.

The method of proof combines several arguments.
First it uses the General Relative Entropy principle introduced by \cite{MMP1,MMP2,BP} for the linear case.
Secondly it reduces the system to a set of ODEs which has the same asymptotic behavior as Equation~\eqref{eq:nonlinear}.
Then we build a Lyapunov functional for this reduced system.
Therefore our result extends several stability results proved for the nonlinear renewal equation in \cite{M3,PTum,Tumuluri2} to the case of the growth-fragmentation equation.

As an immediate consequence of the two theorems, we have the following corollary.

\medskip

\begin{corollary}[Global stability]\label{co:stability}
For $\gamma\in(0,2]$ and under Assumption~\eqref{as:boundfragker}, if $f\in{\mathcal C}^1(\R_+)$ satisfies \eqref{as:f} with ${\mathcal N}$ a singleton,
then there is a unique nontrivial steady state $u_\infty.$
If additionally $f'\left(\int x^pu_\infty(x)\,dx\right)<0,$ then it is globally asymptotically stable in ${\mathcal H}^+\backslash\{0\}.$
\end{corollary}

\

\paragraph{\bf Example 2. Nonlinear drift and death terms.}
We can also treat several nonlinearities as in
\beq\label{eq:nonlinear2}\f{\p}{\p t} u(t,x) = - f\left(\int x^p u(t,x)\,dx\right) \f{\p}{\p x} \bigl(x\,u(t,x)\bigr) - g\left(\int x^q u(t,x)\,dx\right) u(t,x) + {\mathcal F}_\gamma u(t,x).\eeq
In this case, we show that persistent oscillations can appear.
The existence of non-trivial periodic solutions for structured population models is a very interesting and difficult problem.
It has been mainly investigated for age structured models with nonlinear renewal and/or death terms, but there are very few results
\cite{ACHQ,Bertoni,Cushing3,KostovaLi,MagalRuan,MagalRuan2,Pruss1983,Swart,Tumuluri2}.
For Equation~\eqref{eq:nonlinear2}, we exhibit functions $f$ and $g$ for which convergence to periodic solutions can be proved.

Consider differentiable increasing functions $f$ and $g$ such that
\beq\label{as:fandg}f(0)>g(0)=0\qquad \text{and}\qquad f(\infty)<g(\infty)=\infty,\eeq
and which satisfy one of the two following conditions:
\begin{align}
\label{as:example1}&\bullet\quad\exists\, C>0,\ \forall x\geq0,\ g(x)\leq Cxg'(x),\ \text{and}\ f(1)=g(1)=1,\\
\text{or}\quad&\nonumber\\
\label{as:example2}&\bullet\quad\forall x\geq0,\ g(x)=x,\ f\ \text{has a unique fixed point}\ x_0,\ \text{and}\ f'(x_0)<1.\quad
\end{align}
Then we have the following convergence result.

\medskip

\begin{theorem}\label{th:nonlindriftdeath}
Assume that $f$ and $g$ satisfy Assumption~\eqref{as:fandg} and either~\eqref{as:example1} or~\eqref{as:example2},
that $\gamma\in(0,2],$ and that $\kappa$ satifies Assumption~\eqref{as:boundfragker}.
Then there exist parameters $p$ and $q$ for which we can find an open set ${\mathcal V}\subset{\mathcal H}^+$ with the property that
any solution to Equation~\eqref{eq:nonlinear2} with an initial distribution $u_0\in{\mathcal V}$ converges to a nontrivial periodic solution.
\end{theorem}

\medskip

For $f$ and $g$ satisfying~\eqref{as:example1} or~\eqref{as:example2}, there exists a unique nontrivial steady state to Equation~\eqref{eq:nonlinear2}.
For $p$ and $q$ well chosen, we can prove that this steady state is unstable.
Moreover Assumption~\eqref{as:fandg} ensures that the trivial steady state is unstable and that the solutions are bounded.
Then, taking advantage of the Poincar\'e-Bendixon theorem for a reduced ODE system, we can prove for some initial data the convergence to a nontrivial periodic solution.
The details are given in Section~\ref{sec:oscilld2}.

\

\paragraph{\bf Example 3. The prion equation.}
In Section~\ref{sec:prion}, we are interested in a general so-called prion equation
\begin{equation}
\label{eq:Webblin}
\left\{
\begin{array}{rll}
\dfrac{dV(t)}{dt}&=&\displaystyle - V(t) f\left(\int x^pu\right)\int_{0}^{\infty} x\, u(t,x) \; dx - \delta V(t) + \lambda\,,
\vspace{.2cm}\\
\dfrac{\partial}{\partial t} u(t,x) &=& \displaystyle - V(t)f\left(\int x^pu\right) \frac{\partial}{\partial x} \big(x\, u(t,x)\big) - \mu\, u(t,x) + {\mathcal F}_\gamma u(t,x).
\end{array} \right.
\end{equation}
In this equation, the growth term depends on the quantity $V(t)$ of another population (monomers for the prion proliferation model).
We prove for this system the existence of nontrivial periodic solutions under some conditions on $f.$
Define on $\bigl[0,\lambda\mu^{-k-1}\bigr),$ where $k=\f1\gamma,$ the function
\beq\label{def:g}g(x):=\f{\delta\mu}{\lambda-\mu^{k+1}x},\eeq
and consider a positive differentiable function $f$ satisfying
\beq\label{as:fandgprion}\exists!\,x_0>0\quad s.t.\quad f(x_0)=g(x_0)\quad \text{and moreover}\quad 0<f'(x_0)<g'(x_0).\eeq

\medskip

\begin{theorem}\label{th:prion}
Let $f$ a function satisfying Assumption~\eqref{as:fandgprion} and assume that $\mu\leq(k+\mu^{-1})\delta.$
Then there exist parameters $p>0$ for which Equation~\eqref{eq:Webblin} admits nontrivial periodic solutions.
\end{theorem}

\medskip

In age structured models, nontrivial periodic solutions are usually built using bifurcation theory, particularly by Hopf bifurcation (see \cite{LiuMagalRuan} for a general theorem).
Here we use the same method, considering the power $p$ as a bifurcation parameter.
For $f$ satisfying~\eqref{as:fandgprion}, there exists a unique positive steady state for Equation~\eqref{eq:Webblin},
and this steady state undergoes a supercritical Hopf bifurcation when $p$ increases.

\

\paragraph{\bf Example 4. Perron vs. Floquet.}
Our method in Section~\ref{ssec:self-sim} can also be applied to the situation when $V(t)$ and $R(t)$ are periodic controls:
\beq\label{eq:Floquet}\f{\p}{\p t} u(t,x) + V(t)\f{\p}{\p x} \bigl(\tau x\,u(t,x)\bigr) + R(t)\mu\, u(t,x) =\F_\gamma u(t,x).\eeq
In this case, Theorem~\ref{th:nu1} allows to build a particular solution called the Floquet eigenvector, starting from the Perron eigenvector which corresponds to constant parameters.
Moreover, we can compare the associated Floquet eigenvalue to the Perron eigenvalue.
The results about this problem are stated in Section~\ref{sec:Floquet}.

\

Before treating the different examples, we explain in Section~\ref{sec:generalmethod} the general method used to tackle these problems.
It is based on the main result in Theorem~\ref{th:nu1} and the use of the eigenelements of the growth-fragmentation operator together with General Relative Entropy techniques.
In Sections~\ref{sec:nonlineardrift},~\ref{sec:oscilld2},~\ref{sec:prion}, and \ref{sec:Floquet}, we give the proofs of the results in Examples~1,~2,~3,~and~4 respectively.

\

\section{Technical Tools and General Method}\label{sec:generalmethod}

\subsection{Main Theorem}

The proofs of the main theorems of this paper are based on the following result, which requires one to consider that $\tau(x)$ is linear (\emph{i.e.} $\nu=1).$
In this case, there exists a relation between a solution to Equation~\eqref{eq:frag-drift} with time-dependent parameters $(V_1(t),R_1(t))$ and a solution to the same equation with parameters $(V_2(t),R_2(t)).$
More precisely, we can obtain one from the other by an appropriate dilation.
The following theorem generalizes the change of variable used in \cite{EscoMischler3} to build self-similar solution to the pure fragmentation equation.

\medskip

\begin{theorem}\label{th:nu1}
Suppose that Assumptions~\eqref{as:timedependency}-\eqref{as:selfsimfrag} are satisfied with $\nu=1$ and $\gamma>0.$
For $u_1(t,x)$ a solution to Equation~\eqref{eq:frag-drift} with parameters $(V_1,R_1),$ the function $u_2(t,x)$ defined by
\beq\label{eq:trans}u_2(t,x)=W^{-k}(t)u_1\left(h(t),W^{-k}(t)x\right)e^{\,\mu\int_0^t(W(s)R_1(s)-R_2(s))\,ds},\eeq
with $k=\f1\gamma,$ is a solution to Equation~\eqref{eq:frag-drift} with $(V_2,R_2)$ if $W>0$ and $h$ satisfy
\begin{eqnarray}
\dot W&=&\f{\tau W}{k}(V_2-V_1W) \label{eq:WODE1},\\
\dot h&=&W.\nonumber
\end{eqnarray}
Conversely, if $h:\R_+\to\R_+$ is one to one and if $u_2$ is a solution with $(V_2,R_2),$ then $u_1$ defined by~\eqref{eq:trans} is a solution with $(V_1,R_1).$
\end{theorem}

\medskip

The proof of this result is an easy calculation that we leave to the reader.

\medskip

\begin{remark}\label{rk:Bernoulli}
To check that $h$ is one to one, one can take advantage of the fact that ODE~\eqref{eq:WODE1} is a Bernoulli equation which can be integrated:
\beq\label{eq:Bernoulli}W(t)=\f{W_0\,e^{\f\tau k\int_0^t V_2(s)ds}}{1+W_0\f\tau k\int_0^tV_1(s)e^{\f\tau k\int_0^sV_2(s')ds'}\,ds}.\eeq
\end{remark}

\medskip

To tackle the different examples, we use Theorem~\ref{th:nu1} together with two techniques appropriate for this type of equations.
First we recall the existence of particular solutions to the growth-fragmentation equation in the case of time-independent coefficients.
They correspond to eigenvectors of the growth-fragmentation operator and we can give their self-similar dependency on parameters in the case of powerlaw coefficients (see \cite{PG, CDG}).
This dependency is the starting point which leads to Theorem~\ref{th:nu1} and it allows one to build an invariant manifold for Equation~\eqref{eq:frag-drift} in the case $\nu=1.$
It also provides interesting properties on the moments of the solutions when $\nu\neq1.$
Then we recall results about the General Relative Entropy (GRE) introduced by \cite{MMP1,MMP2,BP} for the growth-fragmentation model.
This method ensures that the particular solutions built from eigenvectors are attractive for suitable norms.

\medskip

\subsection{Eigenvectors and Self-similarity: Existence of an Invariant Manifold}
\label{ssec:self-sim}

When the coefficients of Equation~\eqref{eq:frag-drift} do not depend on time, one can build solutions $(t,x)\mapsto\U(x)e^{\Lambda t}$ by solving the Perron eigenvalue problem
\beq
\label{eq:eigenproblem}
\left \{ \begin{array}{l}
\Lambda\,\U(x)=-\f{\p}{\p x} (\tau(x)\,\U(x)) - \mu(x)\,\U(x) + ({\mathcal F}\,\U)(x), \qquad x \geqslant0,
\\
\\
\tau\U(x=0)=0 ,\qquad \U(x)\geq0, \qquad \int_0^\infty \U(x)dx =1.
\end{array} \right.
\eeq
The existence of such elements $\Lambda$ and $\U$ has been first studied by~\cite{M1,PR} and is proved for general coefficients in \cite{DG}.
The dependency of these elements on parameters is of interest to investigate the existence of steady states for nonlinear problems (see \cite{CL1,CDG}).
In the case of powerlaw coefficients, we can work out this dependency on the frozen transport parameter $V$ and the death parameter $R$ (see \cite{PG}).
Under Assumptions~\eqref{as:powerlaw}-\eqref{as:selfsimfrag}, the necessary condition which appears in \cite{DG,M1} to ensure the existence of eigenelements is $\gamma+1-\nu>0.$
Then we define a {\it dilation parameter}
\beq\label{eq:kdef}k:=\f{1}{\gamma+1-\nu}>0,\eeq
and we have explicit self-similar dependencies
\beq\label{eq:VRdependency}
\Lambda(V,R)=V^{k\gamma}\Lambda(1,0)-R\mu\qquad\text{and}\qquad\U(V;x)=V^{-k}\U(1;V^{-k}x).
\eeq
The eigenvector $\U$ does not depend on $R,$ that is why we do not label it.
Hereafter $\Lambda(1,0)$ and $\U(1;\cdot)$ are denoted by $\Lambda$ and $\U$ for the sake of clarity.
The result of Theorem~\ref{th:nu1} is based on the idea to use these dependencies to tackle time-dependent parameters.
An intermediate result between the formula~\eqref{eq:trans} and the dependencies~\eqref{eq:VRdependency} is given by the following corollary.

\medskip

\begin{corollary}\label{co:nu1}
Under the assumptions of Theorem~\ref{th:nu1}, if $W$ is a solution to
\beq\label{eq:WODE2}\dot W=\f{\Lambda(W,0)}{k}(V-W),\eeq
then
\beq\label{eq:autosim}u(t,x)=\U(W(t);x)e^{\int_0^t\Lambda(W(s),R(s))ds}\eeq
is a solution to Equation~\eqref{eq:frag-drift}.
\end{corollary}

\

\begin{proof}
For $\nu=1,$ we can compute $\Lambda=\Lambda(1,0)$ by integrating Equation~\eqref{eq:eigenproblem} with $\mu\equiv0$ against $x\,dx.$
We obtain, due to the {\it mass preservation} of ${\mathcal F},$
$$\Lambda\int_0^\infty x\,\U(x)\,dx=\tau\int_0^\infty x\,\U(x)\,dx,$$
and so $\Lambda(1,0)=\tau.$
Thus, using the dependency~\eqref{eq:VRdependency} we find that $\Lambda(W,0)=\tau W$ and Equation~\eqref{eq:WODE2} is nothing but a rewriting of Equation~\eqref{eq:WODE1}.
We use this formulation~\eqref{eq:WODE2} here to highlight the link with eigenelements, and because it allows one to obtain results in the cases when $\nu\neq1.$
Now we apply Theorem~\ref{th:nu1} for $V_1\equiv1,\ R_1\equiv0$ and $V_2\equiv V,\ R_2\equiv R,$ and we obtain that
$$u_2(t,x)=W^{-k}\U(W^{-k}x)e^{\Lambda h(t)-\int_0^tR(s)\,ds}=\U(W(t);x)e^{\int_0^t\Lambda(W(s),R(s))ds}$$
is a solution to Equation~\eqref{eq:frag-drift}.
\end{proof}

\

This corollary provides a very intuitive explicit solution in the spirit of dependencies~\eqref{eq:VRdependency}.
At each time $t,$ the solution is an eigenvector associated to a parameter $W(t)$ with an instantaneous fitness $\Lambda(W(t),R(t))$ associated to the same parameter $W(t).$
The function $t\mapsto W(t)$ thus defined follows $V(t)$ with a delay explicitely given by ODE~\eqref{eq:WODE2}.

A very useful consequence for the different applications is the existence of an invariant manifold for the growth-fragmentation equation with time-dependent parameters of the form~\eqref{as:timedependency}.
Define the {\it eigenmanifold}
\beq\label{eq:invariantmanifold}{\mathcal E}:=\left\{Q\,\U(W;\cdot),\ (W,Q)\in(\R_+^*)^2\right\},\eeq
which is the union of all the positive eigenlines associated to a transport parameter $W.$
Then Corollary~\ref{co:nu1} ensures that, under the assumptions of Theorem~\ref{th:nu1},
any solution to Equation~\eqref{eq:frag-drift} with an initial distribution $u_0\in{\mathcal E}$ remains in ${\mathcal E}$ for all time.
Moreover the dynamics of such a solution reduces to ODE~\eqref{eq:WODE2}, and this is the key point we use to tackle nonlinear problems.

\

For $\nu\neq1,$ the technique fails and we cannot obtain an explicit solution with the method of Theorem~\ref{th:nu1}.
Nevertheless, we can still define $W$ as the solution to ODE~\eqref{eq:WODE2} and give properties of the functions defined as dilations of the eigenvector by~\eqref{eq:autosim}.
We obtain that the moments of these functions satisfy equations which are similar to the ones verified by the moments of the solution to the growth-fragmentation equation.
In the special case $\nu=0$ and $\gamma=1,$ we even obtain the same equation.
More precisely, if we denote, for $\alpha\geq0,$
\beq\label{eq:momentW}\M_\al[W](t):=\int_0^\infty x^\al \U(W(t);x)e^{\int_0^t\Lambda(W(s),R(s))\,ds}\,dx,\eeq
and
\beq\label{eq:moment}M_\al[u](t):=\int_0^\infty x^\al u(t,x)\,dx,\eeq
then we have the following result.

\medskip

\begin{lemma}\label{lm:moments}
On the one hand, if $W$ is a solution to Equation~\eqref{eq:WODE2}, then the moments $\M_\al$ satisfy 
\beq\label{eq:momentevolW}\dot \M_\al=\al \Lambda a_\al V\M_{\al+\nu-1}+(1-\al)\Lambda b_\al \M_{\al+\gamma}-\mu R\M_\al,\eeq
with
$$a_\al:=\f{M_\al[\U]}{M_{\al+\nu-1}[\U]}\qquad\text{and}\qquad b_\al:=\f{M_\al[\U]}{M_{\al+\gamma}[\U]}.$$
On the other hand, if $u$ is a solution to the fragmentation-drift equation, then the moments $M_\al$ satisfy
\beq\label{eq:momentevol}\dot M_\al=\al\tau V M_{\al+\nu-1}+(c_\al-1)\beta M_{\al+\gamma}-\mu R M_\al,\eeq
with
$$c_\al:=\int_0^1z^\al\kappa(z)\,dz.$$
\end{lemma}

\

\begin{proof}
Using a change of variable, we can compute for all $\al\geq0$
$$\M_\al[W]=M_\al[\U] W^{k\al}e^{\int_0^t\Lambda(W(s),R(s))\,ds},$$
so we have
\begin{eqnarray*}
\dot\M_\al&=&k\al\f{\dot W}{W}\M_\al+\Lambda(W,R)\M_\al\\
&=&\al\Lambda W^{k\gamma-1}(V-W)\M_\al+\Lambda W^{k\gamma}\M_\al-\mu R\M_\al\\
&=&\al\Lambda V W^{k(\nu-1)}\M_\al+(1-\al)\Lambda W^{k\gamma}\M_\al-\mu R\M_\al\\
&=&\al\Lambda V a_\al \M_{\al+\nu-1}+(1-\al)\Lambda b_\al \M_{\al+\gamma}-\mu R \M_\al.
\end{eqnarray*}

Integrating Equation~\eqref{eq:frag-drift} against $x^\al dx,$ we obtain by integration by parts and Fubini's theorem that
\begin{eqnarray*}
\f{d}{dt}\int x^\al u(t,x)\,dx&=& \tau V\int x^{\al}\p_x\bigl(x^{\nu}u(t,x)\bigr)\,dx-\mu R\int x^\al u(t,x)\,dx\\
&&\hspace{.4cm} -\beta \int x^{\al+\gamma}u(t,x)\,dx+\beta \int_0^\infty x^{\al}\int_x^\infty y^{\gamma-1}\kappa\left(\f xy\right)\,dy\,dx\\
&=&\al\tau V\int x^{\al+\nu-1}u(t,x)\,dx-\mu R\int x^\al u(t,x)\,dx\\
&&\hspace{.4cm} -\beta \int x^{\al+\gamma}u(t,x)\,dx+\beta \int_0^\infty y^{\al+\gamma}\int_0^y \f{x^\al}{y^\al}\kappa\left(\f xy\right)\,\f{dx}{y}\,dy\\
&=&\al\tau VM_{\al+\nu-1}+(c_\al-1)\beta M_{\al+\gamma}-\mu R M_\al.
\end{eqnarray*}

\end{proof}

\

In the particular case $\nu=0$, $\gamma=1,$ and $\kappa$ is symmetric, we can compute the constants $a_\al,$ $b_\al,$ and $c_\al$ for $\al=1$ or $\al=2.$
A consequence is the useful result given in the following corollary.

\medskip

\begin{corollary}\label{co:nu0gamma1}
In the case when $\nu=0,\ \gamma=1,$ and $\kappa$ is symmetric, the zero and first moments $(\M_0,\M_1)$ and $(M_0,M_1)$ are both solutions to
\beq\label{eq:ODEclosed}
\left(\begin{array}{c}
\dot U\\
\dot P
\end{array}\right)
=
\left(\begin{array}{cc}
-\mu R&\beta\\
\tau V&-\mu R
\end{array}\right)
\left(\begin{array}{c}
U\\
P
\end{array}\right).
\eeq
\end{corollary}

\medskip

This Corollary will allow us, in Section~\ref{sec:Floquet}, to compare the Perron and Floquet eigenvalues not only for $\nu=1$ but also for $\nu=0,\ \gamma=1.$
In this case we do not have a particular solution to the growth-fragmentation equation as in Corollary~\ref{co:nu1}, but a particular solution of the reduced ODE system~\eqref{eq:ODEclosed}.

\medskip

\begin{proof}
For $\kappa$ symmetric, we have already seen that $c_0=n_0=2.$
Together with Assumption~\eqref{as:b} which gives $c_1=1,$ we conclude that $(M_0,M_1)$ is solution to Equation~\eqref{eq:ODEclosed}.

Integrating Equation~\eqref{eq:eigenproblem} against $dx$ and $x\,dx$ we obtain
$$\Lambda=\beta\int x\,\U(x)\,dx\qquad\text{and}\qquad\Lambda\int x\,\U(x)\,dx=\tau.$$
This allows to compute
$$\int x\,\U(x)\,dx=\sqrt{\frac{\tau}{\beta}}\qquad\text{and}\qquad \Lambda=\sqrt{\tau\beta}.$$
Thus $a_1=\sqrt{\frac{\tau}{\beta}}$ and $b_0=\sqrt{\frac{\beta}{\tau}},$ and $(\M_0,\M_1)$ satisfies Equation~\eqref{eq:ODEclosed} due to Lemma~\ref{lm:moments}.
\end{proof}

\medskip

\subsection{General Relative Entropy: Attractivity of the Invariant Manifold}
\label{ssec:GRE}

The existence of the invariant manifold ${\mathcal E}$ is useful to obtain particular solutions to nonlinear growth-fragmentation equations.
But what happens when the initial distribution $u_0$ does not belong to this manifold?
The GRE method ensures that ${\mathcal E}$ is attractive in a sense to be defined.

\

The GRE method requires one to consider the adjoint growth-fragmentation equation 
\beq\label{eq:adj-frag-drift}-\f{\p}{\p t} \psi(t,x) = \tau(t,x) \f{\p}{\p x} \psi(t,x) - \mu(t,x)\psi(t,x) + ({\mathcal F}^*\psi)(t,x),\eeq
where ${\mathcal F}^*$ is the adjoint fragmentation operator
$$({\mathcal F}^*\psi)(t,x):=\int_0^x b(t,x,y)\psi(t,y)\,dy-\beta(t,x)\psi(t,x).$$
If $u$ and $v$ are two solutions to Equation~\eqref{eq:frag-drift} and $\psi$ is a solution to Equation~\eqref{eq:adj-frag-drift}, then we have, for any function $H:\R\to\R,$
\begin{align*}
\f{d}{dt}\int_0^\infty \psi(t,&x)v(t,x)\,H\left(\f{u(t,x)}{v(t,x)}\right)\,dx=\\
-&\int_0^\infty\int_y^\infty b(t,y,x)\psi(t,x)v(t,y)\\
&\times\left[H\left(\f{u(t,x)}{v(t,x)}\right)-H\left(\f{u(t,y)}{v(t,y)}\right)+H'\left(\f{u(t,x)}{v(t,x)}\right)\left(\f{u(t,y)}{v(t,y)}-\f{u(t,x)}{v(t,x)}\right)\right]\,dxdy.\end{align*}
When $H$ is convex, the right hand side is nonpositive and we obtain a nonincreasing quantity called GRE.

\

In the case of time-independent coefficients, we can choose for $v$ a solution of the form $\U(x)e^{\Lambda t}.$
Then, to apply the GRE method, we need a solution to the adjoint equation; such solutions are given by solving the adjoint Perron eigenvalue problem
\beq
\label{eq:adjeigenproblem}
\left \{ \begin{array}{l}
\Lambda\phi(x)=\tau(x) \f{\p}{\p x} (\phi(x)) - \mu(x) \phi(x) + (\F^*\phi)(x), \qquad x \geqslant0,
\\
\\
\phi(x)\geq0, \qquad \int_0^\infty \phi(x)\U(x)dx =1.
\end{array} \right.
\eeq
Such a problem is usually solved together with the direct problem~\eqref{eq:eigenproblem} and the first eigenvalue $\Lambda$ is the same for each problem (see \cite{DG,M1,BP}).
Then the GRE ensures that any solution $u$ to the growth-fragmentation equation behaves asymptotically like $\U(x)e^{\Lambda t}.$
More precisely it is proved in \cite{MMP2,BP} under general assumptions that
\beq\label{eq:convcst}\lim_{t\to\infty}\|\varrho_0^{-1}u(t,\cdot)e^{-\Lambda t}-\U\|_{L^p(\U^{1-p}\phi\,dx)}=0,\qquad\forall p\geq1,\eeq
where $\varrho_0=\int u_0(y)\phi(y)\,dy$ with $u_0(x)=u(t=0,x).$

\

Now consider Equation~\eqref{eq:frag-drift} whose coefficients are time-dependent.
Under the assumptions of Theorem~\ref{th:nu1} and for $p=1,$
the convergence result~\eqref{eq:convcst} can be interpreted as the attractivity of the invariant manifold ${\mathcal E}$ in $L^1(\phi\,dx)$ with the distance
$$d(u,{\mathcal E}):=\inf_{\U\in{\mathcal E}}\|\varrho^{-1}u-\U\|_{L^1(\phi(x)\,dx)},$$
where $\varrho:=\int u(y)\phi(y)\,dy.$
Consider, for $V(t)\geq0,$ a solution $W$ to
$$\dot W=\f Wk(\tau V-W)$$
with $W(0)=1.$
First we have $\dot W\geq-\f1k W^2,$ so $W\geq\f1{1+\f tk}$ and $h\geq k \ln(1+\f tk).$
Thus $h:\R_+\to\R_+$ is one to one and we can build from a solution $u(t,x)$ to Equation~\eqref{eq:frag-drift} the function
\beq\label{eq:v}v(h(t),x):=W^k(t)\,u(t,W^k(t)x)\,e^{-\int_0^t(W(s)-\mu R(s))\,ds},\eeq
which satisfies $v(0,x)=u(0,x)=u_0(x).$
Using Theorem~\ref{th:nu1}, $v(t,x)$ is a solution to
\beq\label{eq:linear2}\f{\p}{\p t}v(t,x)=-\f{\p}{\p x}\bigl(x\,v(t,x)\bigr)-v(t,x)+(\F_\gamma v)(t,x)\eeq
so, denoting by $\U$ and $\phi$ the eigenfunctions of Equation~\eqref{eq:linear2}, we have
$$\int_0^\infty v(t,x)\phi(x)\,dx=\int_0^\infty v(t=0,x)\phi(x)\,dx=\varrho_0$$
and
$$\lim_{t\to\infty}\|\varrho_0^{-1}v(t,\cdot)-\U\|_{L^1(\phi(x)\,dx)}=0.$$
For $\nu=1,$ $\phi$ is linear (see examples in \cite{DG}), so we can compute from~\eqref{eq:v}
$$\varrho(t)=\int u(t,y)\phi(y)\,dy=\varrho_0 W^k(t)\,e^{\int_0^t(W(s)-\mu R(s))\,ds}$$
and
\beq\label{eq:Eattract}d(u(t,\cdot),{\mathcal E})\leq\|\varrho^{-1}(t)u(t,\cdot)-W^{-k}\U(W;\cdot)\|=\|\varrho_0^{-1}v(h(t),\cdot)-\U\|\to0.\eeq
This is what we call the attractivity of ${\mathcal E}.$

\

The exponential decay in \eqref{eq:convcst} is proved in \cite{PR,LP} for $p=1$ and for a constant fragmentation rate $\beta(x)\equiv\beta.$
It is also proved in \cite{CCM} for powerlaw parameters in the norm corresponding to $p=2$ and this is the case we are interested in.
A spectral gap result is proved in $L^2(\U^{-1}\phi\,dx)$ and the result is extended to bigger spaces thanks to a general method for spectral gaps in Hilbert spaces.

\medskip

\begin{theoremCCM}
Under Assumption~\eqref{as:powerlaw} with $\nu=1$, Assumption~\eqref{as:selfsimfrag} with $\gamma\in(0,2],$ and Assumption~\eqref{as:boundfragker},
there exist $\bar a>0$ and $\bar r\geq3$ such that, for any $a\in(0,\bar a)$ and any $r\geq\bar r,$ there exists $C_{a,r}$ such that for any $u_0\in{\mathcal H}:=L^2(\theta),\ \theta(x)=x+x^r,$ there holds
\beq\label{eq:spectralgap}\forall t>0,\qquad\|\varrho_0^{-1}u(t,\cdot)e^{-\Lambda t}-\U\|_{\mathcal H}\leq C_{a,r}\|\varrho_0^{-1}u_0-\U\|_{\mathcal H}\,e^{-at}.\eeq
\end{theoremCCM}

\medskip

This result is very useful to treat Examples 1 and 2 because $L^2(\theta)\subset L^1(x^p)$ for $r\geq 2p+1.$
Moreover the exponential decay allows one to prove exponential stability results for Equation~\eqref{eq:nonlinear} when $\kappa$ is constant (see Section~\ref{sec:nonlineardrift}).

\

\section{Nonlinear Drift Term: Convergence and Stability}\label{sec:nonlineardrift}

Consider the nonlinear growth-fragmentation equation~\eqref{eq:nonlinear} where the transport term depends on the $pth$-moment of the solution itself.
This dependency may represent the influence of the total population of individuals on the growth process of each individual.
We study the long-time asymptotic behavior of the solutions in the positive cone ${\mathcal H}^+$ with the weight $\theta(x)=x+x^r$ for
\beq\label{as:r}r\geq\max(\bar r,2p+1).\eeq
We prove that there is always convergence to a steady state, provided that the function $f$ is less than $\mu$ at the infinity.
This result, stated in Theorem~\ref{th:nonlindrift:convergence} in the introduction, is made precise in Theorem~\ref{th:convergence} with the expression of the steady states in terms of eigenfunctions.
Results about their stability are given in Theorem~\ref{th:stability} and proved in the current section.
We use the notation $M_p$ for $M_p[\U]=\int x^p\U(x)\,dx.$

\medskip

\begin{theorem}\label{th:convergence}
Assume that $f$ is a continuous function on $[0,+\infty)$ which satisfies Assumption~\eqref{as:f}, that $\gamma\in(0,2],$ and that the fragmentation kernel $\kappa$ satisfies Assumption~\eqref{as:boundfragker}.
Then the nontrivial steady states of Equation~\eqref{eq:nonlinear} are
$$\f{I_\infty}{M_p}\mu^{-kp}\,\U(\mu;\cdot)\quad with \quad I_\infty\in{\mathcal N},$$
and for any solution $u,$ there exists $I_\infty\in{\mathcal N}\cup\{0\}$ such that
\beq\label{eq:convergence}\lim_{t\to\infty}\left\|u(t,\cdot)-\f{I_\infty}{M_p}\mu^{-kp}\,\U(\mu;\cdot)\right\|_{\mathcal H}=0.\eeq
\end{theorem}

\

\begin{proof}[Proof of Theorem~\ref{th:convergence}]
\noindent{\bf First step: $u_0\in{\mathcal E}.$}\\
Consider an initial distribution $u_0\in{\mathcal E}$ defined in \eqref{eq:invariantmanifold}. Then there exist $W_0>0$ and $Q_0\geq0$ such that
$$u_0(x)=Q_0\,\U(W_0\mu;x).$$
Let $u(t,x)$ be the solution to Equation~\eqref{eq:nonlinear} and define $W$ as the solution to
\begin{equation}\label{eq:redODE}
\left\{\begin{array}{cll}
\dot W&=&\dis\f Wk\left(f\left(\int x^pu(t,x)\,dx\right)- \mu W\right),
\vspace{.2cm}\\
\dot W(0)&=&W_0.
\end{array}\right.
\end{equation}
Then Corollary~\eqref{co:nu1} ensures that
$$\forall\,t,x\geq0,\quad u(t,x)=Q_0\,\U(W(t)\mu;x)e^{\mu\int_0^t(W(s)-1)\,ds},$$
and so we have
$$\int_0^\infty x^pu(t,x)\,dx=Q_0W^{kp}\mu^{kp}\left(\int_0^\infty x^p\U(x)\,dx\right)e^{\mu\int_0^t(W(s)-1)\,ds}.$$
Now defining
$$Q(t):=Q_0 e^{\mu\int_0^t(W(s)-1)\,ds},$$
we obtain a system of ODEs equivalent to Equation~\eqref{eq:nonlinear} in ${\mathcal E},$
\begin{equation}\label{eq:nonlinODE}
\left\{\begin{array}{rcl}
\dot W&=&\dis\f Wk\left(f_p\left(W^{kp}Q\right)- \mu W\right),
\vspace{.2cm}\\
\dot Q&=&\mu\,Q(W-1),
\end{array}\right.
\end{equation}
with the notation
$$f_p(I):=f\left(I\mu^{kp}\int_0^\infty x^p\U(x)\,dx\right),$$
and proving convergence of $u$ is equivalent to proving convergence of $(W,Q).$
To study System~\eqref{eq:nonlinODE}, we change the unknown to $Z:=W^{kp}Q.$
Then $(W,Z)$ is solution to
\begin{equation}\label{eq:ODEWZ}
\left\{\begin{array}{rcl}
\dot W&=&\dis\f Wk\left(f_p\left(Z\right)- \mu W\right),
\vspace{.2cm}\\
\dot Z&=&pZ\left(f_p\left(Z\right)- \mu W\right)+ \mu Z(W-1),
\end{array}\right.
\end{equation}
and the positive steady states satisfy $W_\infty=\mu^{-1}f_p(Z_\infty)=1.$
To prove convergence to one of these steady states, we exhibit a Lyapunov functional.
Denoting $A:=\mu(W-1)$ and $B:=\mu-f_p(Z),$ System~\eqref{eq:ODEWZ} writes
\begin{equation}\label{eq:ODEAB}
\left\{\begin{array}{rcl}
k\dfrac{\dot W}{W}&=&-A-B,
\vspace{.2cm}\\
\dfrac{\dot Z}{Z}&=&-(p-1)A-pB.
\end{array}\right.
\end{equation}
For $\al>0,$ multiply the first equation by $\al A$ and the second by $B.$
The sum gives
\begin{align*}
\al k \f{\dot W}{W}A+\f{\dot Z}{Z}B&=-\Bigl(\al A^2+pB^2+(\al+p-1)AB\Bigr)\\
&=-\left(\al A^2+pB^2+\f{\al+p-1}{2\sqrt{\al p}}2\sqrt{\al p}AB\right)\\
&\leq \f{|\al+p-1|}{2\sqrt{\al p}}\left(\al A^2+pB^2\right)-\left(\al A^2+pB^2\right)\\
&\leq\left(\f{|\al+p-1|}{2\sqrt{\al p}}-1\right)\left(\al A^2+pB^2\right)\\
&\leq -\omega\left(\al A^2+pB^2\right)
\end{align*}
with $\omega>0$ if we can find $\al$ such that $|\al+p-1|<2\sqrt{\al p}.$
Equivalently we have to find $\al>0$ such that
\begin{align}
(\al+p-1)^2&<4\al p,\nonumber\\
\al^2-2(p+1)\al+&(p-1)^2<0.\label{eq:binomial}
\end{align}
But the reduced discriminant is
$$\Delta'=(p+1)^2-(p-1)^2=4p>0$$
and there always exists $\al>0$ such that~\eqref{eq:binomial} is satisfied.
Finally, defining
$$G(W):= W-1-\ln(W)$$
and
$$F(Z):=\int_1^Z(\mu-f_p(z))\f{dz}{z},$$
we obtain that $L(W,Z):=\al k\mu G(W)+F(Z)$ is a Lyapunov functional for System~\eqref{eq:ODEWZ}.
Indeed it satisfies
$$\f{d}{dt}L(W(t),Z(t))=\al k \f{\dot W}{W}A+\f{\dot Z}{Z}B\leq-\omega(\al A^2+B^2):=-D(W,Z),$$
with $D(W,Z)$ positive out of the steady states,
and Assumption~\eqref{as:f} ensures that $L$ and $D$ are coercive in the sense that $L(W,Z)\to+\infty$ and $D(W,Z)\to+\infty$ when $\|(W,Z)\|\to+\infty.$
So we can infer the convergence of the solution to a steady state.
If $f(0)>\mu,$ $L(W,Z)\to+\infty$ when $W$ or $Z$ tends to $0,$ so for any $(W_0>0,Z_0>0)$ the solution $(W,Z)$ converges to a critical point of $L,$ namely $(1,\f{I_\infty}{\mu^{kp}M_p})$ with $I_\infty\in{\mathcal N}.$
If $f(0)<\mu,$ then for any $\bar W>0$ we have that $L(W,Z)\xrightarrow{(W,Z)\to(\bar W,0)}-\infty.$
So either $(W,Z)$ converges to $(1,\f{I_\infty}{\mu^{kp}M_p})$ with $I_\infty\in{\mathcal N},$ or $Z\to0.$
To obtain the convergence in ${\mathcal H},$ we write
$$\|u(t,\cdot)-Z_\infty\U(\mu;\cdot)\|=\int_0^\infty(Q(t)\U(W(t)\mu;x)-Z_\infty\U(\mu;x))^2(x+x^r)\,dx,$$
and we use dominated convergence.
We know from Theorem~1 in \cite{DG} that under Assumption~\eqref{as:boundfragker} and for $\nu=1,$ $x^\al\U(x)$ is bounded in $\R_+$ for all $\al\geq0,$ so it ensures that the integrand is dominated by an integrable function.
Then the convergence in~${\mathcal H}$ is given by the convergence of $(W,Z),$ so convergence~\eqref{eq:convergence} occurs.

\

\noindent{\bf Second step: general initial distribution $u_0.$}\\
We now assume that $u$ a solution to Equation~\eqref{eq:nonlinear} not necessarily in ${\mathcal E},$ and we define as in Section~\ref{ssec:GRE}
$$v(h(t),x):=W^{k}(t)u(t,W^{k}(t)x)e^{\mu(t-h(t))},$$
with
$$\dot W=\f W k\left(f\left(\int x^p u\right)-\mu W\right)$$
and
$$\dot h=W.$$
We recall that in this case $h:\R_+\to\R_+$ is one to one.
We choose the initial values $W(0)=1$ and $h(0)=0$ to have $v(0,x)=u(0,x)=u_0(x).$
Then $v(t,x)$ is a solution to
\beq\label{eq:linear}\f{\p}{\p t} v(t,x) = - \mu\f{\p}{\p x} \big(x v(t,x)\big) - \mu v(t,x) + {\mathcal F}_\gamma v(t,x)\eeq
and, due to the GRE, we conclude that
$$v(t,x)\xrightarrow[t\to\infty]{}\varrho_0\,\U(\mu;x),$$
where
$$\varrho_0=\int_0^\infty\phi(\mu;x)v_0(x)\,dx=\int_0^\infty\phi(\mu;x)u_0(x)\,dx,$$
and so
$$\int x^p u(t,x)\,dx\underset{t\to\infty}{\sim}\varrho_0\mu^{kp}\left(\int x^p\U(x)\,dx\right)W^{kp}(t)e^{\mu(h(t)-t)}dx.$$
As a consequence it holds that
$$\dot W(t)\underset{t\to\infty}{\sim}\f Wk\left(f_p\left(W^{kp}Q\right)-\mu W\right)$$
with $Q(t)=\varrho_0\,e^{\mu(h(t)-t)}$ satisfying the equation
$$\dot Q=\mu Q(W-1).$$
The interpretation of this is that System~\eqref{eq:redODE} represents asymptotically the dynamics of the solutions to Equation~\eqref{eq:nonlinear}.
More rigorously, define
\beq\label{eq:defep}\e(t):=\frac{\int x^p u(t,x)\,dx}{\mu^{kp}M_pQ(t)W^{kp}(t)}-1.\eeq
Then we have
$$\dot W(t)=\f Wk\left(f_p\left(W^{kp}Q(1+\e(t))\right)-\mu W\right)$$
and, using the Cauchy-Schwarz inequality and the exponential decay theorem of \cite{CCM},
\begin{eqnarray}\label{eq:ep}
|\e(t)|&=&\frac{W^{-kp}e^{\mu (t-h(t))}}{\mu^{kp}M_p}\left|\int(\varrho_0^{-1}u(t,x)-\U(\mu;x) W^{kp}e^{\mu (h(t)- t)})x^p\,dx\right|\nonumber\\
&=&\frac{1}{\mu^{kp}M_p}\left|\int(\varrho_0^{-1}v(h(t),x)-\U(\mu;x))x^p\,dx\right|\nonumber\\
&\leq&\frac{1}{\mu^{kp}M_p}\left(\int|\varrho_0^{-1}v(h(t),x)-\U(\mu;x)|^2(x+x^r)\,dx\right)^{\frac12}\left(\int \frac{x^{2p}}{x+x^r}dx\right)^{\frac12}\nonumber\\
&&\vspace{.2cm}\nonumber\\
&\leq& \frac{C_p}{\mu^{kp}M_p}\|\varrho_0^{-1}u_0-\,\U(\mu;\cdot)\|e^{-ah(t)}.
\end{eqnarray}
The function $\f{x^{2p}}{x+x^r}$ is integrable under Assumption~\eqref{as:r}.
Finally, the long-time dynamics of the solution $u$ to Equation~\eqref{eq:nonlinear} is prescribed by
\begin{equation}\label{eq:ODEWe}
\left\{\begin{array}{rcl}
k\dfrac{\dot W}{W}&=&-\mu(W-1)-\bigl(\mu-f_p((1+\e)Z)\bigr),
\vspace{.2cm}\\
\dfrac{\dot Z}{Z}&=&-(p-1)\mu (W-1)-p\bigl(\mu-f_p((1+\e)Z)\bigr),
\end{array}\right.
\end{equation}
where $Z=W^{kp}Q,$ and $\e(t)\xrightarrow[t\to\infty]{}0.$
Now we see what becomes the Lyapunov functional of the first step for this system to obtain
\beq\label{eq:Lyapunovep}\f{d}{dt}L(W,Z)\leq-D(W,Z)+\underbrace{(\al A+pB)\bigl(f_p((1+\e)Z)-f_p(Z)\bigr)}_{:=E(W,Z,\e)}.\eeq
Thanks to Assumption~\eqref{as:f} we know that $f_p$ is bounded; hence $W,$ which is a solution to
$$\dot W=\f Wk\bigl(f_p((1+\e)Z)-\mu W\bigr),$$
is also bounded.
Thus $E(W,Z,\e)$ is bounded and, because $L$ and $D$ are coercive, the trajectory $(W,Z)$ is bounded.
Moreover $E(W(t),Z(t),\e(t))\xrightarrow[t\to\infty]{}0$ because $\e(t)\to0,$ so $(W,Z)$ converges to a steady state $(1,Z_\infty).$
Finally we write
$$\|u(t,\cdot)-Z_\infty\U(\mu;\cdot)\|\leq\|u(t,\cdot)-Q(t)\U(W(t)\mu;\cdot)\|+\|Q(t)\U(W(t)\mu;\cdot)-Z_\infty\U(\mu;\cdot)\|,$$
and we know from the first step that $\|Q(t)\U(W(t)\mu;\cdot)-Z_\infty\U(\mu;\cdot)\|\to0.$
We treat the other term using the spectral gap theorem of \cite{CCM}:
\begin{eqnarray}\label{eq:u-QU}
\|u(t,\cdot)-Q(t)\U(W(t)\mu;\cdot)\|_{\mathcal H}&=&Q\left(\int|\varrho_0^{-1}v(h(t),x)-\U(\mu;x)|^2(x+W^{(r-1)k}x^r)\,dx\right)^{\f12}\nonumber\\
&\leq&C\|\varrho_0^{-1}v(h(t),\cdot)-\U(\mu;\cdot)\|\nonumber\\
&\leq&C\|\varrho_0^{-1}u_0-\U(\mu;\cdot)\|\,e^{-ah(t)}. 
\end{eqnarray}
So $\|u(t,\cdot)-Q(t)\U(W(t)\mu;\cdot)\|\to0$ because $h(t)\to+\infty,$ and the proof is complete.
\end{proof}

\

\begin{proof}[Proof of Theorem~\ref{th:stability}]
\noindent{\bf Stability of the trivial steady state.}\\
We start with the stability of the zero steady state when $f(0)<\mu$ and $p\geq1.$
For this we integrate Equation~\eqref{eq:nonlinear} against $x^p$ and find
$$\f{d}{dt}\left(\int x^pu(t,x)\,dx\right)\leq \left(f\left(\int x^pu(t,x)\,dx\right)-\mu\right)\left(\int x^pu(t,x)\,dx\right)$$
due to the mass conservation Assumption \eqref{as:b}.
Thus $\int x^pu(t,x)\,dx$ is a decreasing function if $f\left(M_p[u_0]\right)<\mu.$
Then we integrate against $x+x^r$ for any $r\geq1$ and we obtain
$$\f{d}{dt}\left(\int u(t,x)\,(x+x^r)dx\right)\leq \left(f\left(\int x^pu(t=0,x)\,dx\right)-\mu\right)\left(\int u(t,x)\,(x+x^r)dx\right),$$
which ensures the exponential convergence
$$\|u(t,\cdot)\|_{\mathcal H}\leq \|u_0\|_{\mathcal H}\,e^{\left(f\left(M_p[u_0]\right)-\mu\right)t}.$$
For $r\geq p,$ we have $M_p[u_0]\leq C \|u_0\|_{\mathcal H},$ so for $\|u_0\|$ small enough we have $f(M_p[u_0])<\mu,$ and the exponential convergence occurs.

When $f(0)>\mu,$ we have seen in the proof of Theorem~\ref{th:convergence} that $L(W,Z)\to+\infty$ when $W$ or $Z$ tends to zero.
So the trivial steady state is unstable.

\

\noindent{\bf Stability of nontrivial steady states.}\\
Let $(W_\infty,Z_\infty)$ be a positive steady state to System~\eqref{eq:ODEWZ}.
We want to prove that $Z_\infty\U(\mu;\cdot)$ is locally asymptotically stable.
Since Theorem~\ref{th:convergence} ensures the convergence of any solution to Equation~\eqref{eq:nonlinear} toward a steady state,
it only remains to prove the local stability of $Z_\infty\U(\mu;\cdot),$ namely
$$\forall\,\rho_1>0,\ \exists\,\rho_2>0,\ \|u_0-Z_\infty\U(\mu;\cdot)\|<\rho_2\quad \Rightarrow\quad \forall\,t>0,\ \|u(t,\cdot)-Z_\infty\U(\mu;\cdot)\|<\rho_1.$$
We have already seen that
$$\|u(t,\cdot)-Z_\infty\U(\mu;\cdot)\|\leq C\|\varrho_0^{-1}u_0-\U(\mu;\cdot)\|+\|Q(t)\U(W(t)\mu;\cdot)-Z_\infty\U(\mu;\cdot)\|,$$
with $\|Q\U(W\mu;\cdot)-Z_\infty\U(\mu;\cdot)\|\to0$ when $(W,Z)\to(W_\infty,Z_\infty).$

Let first treat the term $\|\varrho_0^{-1}u_0-\U(\mu;\cdot)\|.$
We have
\begin{eqnarray}\label{eq:Z-norm}
|\varrho_0-Z_\infty|&=&\left|\int (u_0-Z_\infty\U(\mu;x))\phi(\mu;x)\,dx\right|\nonumber\\
&\leq&\mu^{-k}\left(\int (u_0-Z_\infty\U(\mu;x))^2(x+x^r)\,dx\right)^{\f12}\left(\int\f{x^2}{x+x^r}\,dx\right)^{\f12}\nonumber\\
&=&C\|u_0-Z_\infty\U(\mu;\cdot)\|,
\end{eqnarray}
and also
\begin{eqnarray}\label{eq:norm-norm}
\|\varrho_0^{-1}u_0-\U(\mu;\cdot)\|&\leq&\varrho_0^{-1}\|u_0-\f{I_\infty}{\mu^{kp}M_p}\U(\mu;\cdot)\|+\varrho_0^{-1}|\varrho_0-Z_\infty|\|\U(\mu;\cdot)\|\nonumber\\
&\leq&\f{C}{\varrho_0}\|u_0-Z_\infty\U(\mu;\cdot)\|\nonumber\\
&\leq&\f{C}{Z_\infty-|\varrho_0-Z_\infty|}\|u_0-Z_\infty\U(\mu;\cdot)\|\nonumber\\
&\leq&\f{C\|u_0-Z_\infty\U(\mu;\cdot)\|}{Z_\infty-C\|u_0-Z_\infty\U(\mu;\cdot)\|},
\end{eqnarray}
so $\|\varrho_0^{-1}u_0-\U(\mu;\cdot)\|$ is small for $\|u_0-Z_\infty\U(\mu;\cdot)\|$ small enough.

Now let turn to the term $\|Q\U(W\mu;\cdot)-Z_\infty\U(\mu;\cdot)\|.$
Since it tends to zero when $(W,Z)$ tends to $(W_\infty,Z_\infty),$ it is sufficient to prove that $(W_\infty,Z_\infty)$ is stable for System~\eqref{eq:ODEWe}.
For $\eta>0,$ denote by ${\mathcal V}_\eta(W_\infty,Z_\infty)$ the connected component of $\{(W,Z),\ L(W,Z)<L(W_\infty,Z_\infty)+\eta\}$ which contains $(W_\infty,Z_\infty).$
Since $f'(I_\infty)<0,$ $(W_\infty,Z_\infty)$ is a strict local minimum of $L$ so we have, denoting $B(X,\rho)=\{Y\in\R^2,\ \|X-Y\|<\rho\},$
$$\forall\,\rho>0,\ \exists\,\eta>0,\quad {\mathcal V}_\eta(W_\infty,Z_\infty)\subset B((W_\infty,Z_\infty),\rho),$$
and reciprocally
$$\forall\,\eta>0,\ \exists\,\rho>0,\quad B((W_\infty,Z_\infty),\rho)\subset{\mathcal V}_\eta(W_\infty,Z_\infty).$$
So it is sufficient, to have the local stability of $(W_\infty,Z_\infty),$ to prove that ${\mathcal V}_\eta(W_\infty,Z_\infty)$ is stable.
This is true for System~\eqref{eq:ODEWZ} since in this case $L$ is a Lyapunov functional.
Then, by continuity of $f,$ there exists $\e_\eta$ such that ${\mathcal V}_\eta(W_\infty,Z_\infty)$ remains stable for System~\eqref{eq:ODEWe} if $|\e(t)|<\e_\eta$ for all $t>0.$
But we know from~\eqref{eq:ep} that $|\e(t)|\leq C\|\varrho_0^{-1}u_0-\U(\mu;\cdot)\|$
and from~\eqref{eq:norm-norm} that $\|\varrho_0^{-1}u_0-\U(\mu;\cdot)\|$ is small for $\|u_0-Z_\infty\U(\mu;\cdot)\|$ small enough.
Finally $\|Q\U(W\mu;\cdot)-Z_\infty\U(\mu;\cdot)\|$ is small for $\|u_0-Z_\infty\U(\mu;\cdot)\|$ small enough and the local asymptotical stability of the nontrivial steady states is proved.

\

Now we assume that $\kappa\equiv2$ and prove the local exponential stability of $Z_\infty\U(\mu;\cdot).$
In this case we have (see examples in \cite{DG}) the explicit formula
\beq\label{eq:explicitU}\U(x)=C\,e^{-\f\beta\gamma x^\gamma},\eeq
and due to this we can estimate the quantity $\|Q\U(W\mu;\cdot)-Z_\infty\U(\mu;\cdot)\|.$
We have
\begin{align*}
\|Q\,\U(W\mu;\cdot)-Z_\infty\U(\mu;\cdot)\|^2&=\int|Q\,\U(W\mu;x)-Z_\infty\U(\mu;x)|^2(x+x^r)\,dx\\
&\leq C|Z-Z_\infty|^2W^{-2kp}\int\U(W\mu;x)^2(x+x^r)\,dx\hspace{1.7cm}(i)\\
&\hspace{.2cm}+C|W^{-k(p+1)}-1|^2Z_\infty^2\int\U(W^{-k}\mu^{-k}x)^2(x+x^r)\,dx\hspace{.4cm}(ii)\\
&\hspace{.4cm}+CZ_\infty^2\int|\U(W^{-k}\mu^{-k}x)-\U(\mu^{-k}x)|^2(x+x^r)\,dx\hspace{.6cm}(iii)
\end{align*}
and we prove exponential decay of $(i),$ $(ii),$ and $(iii)$ in a neighbourhood of $(W_\infty,Z_\infty).$
We have, due to the L'H\^opital rule,
\beq\label{eq:hopitalF}F(Z)\underset{Z\to Z_\infty}{\sim}\frac{1}{-2Z_\infty f'(Z_\infty)} (\mu-f(Z))^2\eeq
and
\beq\label{eq:hopitalG}G(W)\underset{W\to W_\infty}{\sim}\f{1}{2W_\infty}(W-1)^2,\eeq
so the following local Poincar\'e inequality holds:
\beq\label{eq:entropy-entropy}\exists\,b>0,\ \rho>0,\quad \forall(W,Z)\in B((W_\infty,Z_\infty),\rho),\quad L(W,Z)\leq \f{1}{4b} D(W,Z).\eeq
Fix such a $\rho<1$ and fix $\eta>0$ such that ${\mathcal V}_\eta(W_\infty,Z_\infty)\subset B((W_\infty,Z_\infty),\rho).$
Consider $\|u_0-Z_\infty\U(\mu;\cdot)\|$ small enough so that $|\e(t)|$ remains smaller than $\e_\eta$ for all time.
Then ${\mathcal V}_\eta(W_\infty,Z_\infty)$ is stable for the dynamics of System~\eqref{eq:ODEWe}.
Now look at the term $E(W(t),Z(t),\e(t))$ for a solution $(W,Z)$ to System~\eqref{eq:ODEWe} in this stable neighbourhood.
It satisfies
\begin{eqnarray*}
E(W,Z,\e)&=&(\al A+pB)\bigl(f_p((1+\e)Z)-f_p(Z)\bigr)\\
&=&\sqrt{\omega\alpha}A\,\sqrt{\f\alpha\omega}\bigl(f_p((1+\e)Z)-f_p(Z)\bigr)+\sqrt{\omega p}B\,\sqrt{\f p\omega}\bigl(f_p((1+\e)Z)-f_p(Z)\bigr)\\
&\leq& \f\omega2(\alpha A^2+pB^2)+\f{\al+p}{2\omega}\bigl(f_p((1+\e)Z)-f_p(Z)\bigr)^2\\
&=& \f12 D(W,Z) + C\,(f_p((1+\e)Z)-f_p(Z))^2\\
&\leq&\f12 D(W,Z) + C\sup_{J}|f'|\,\e^2,
\end{eqnarray*}
where $J=[Z_\infty-\rho-\e_\eta,Z_\infty+\rho+\e_\eta].$
As a consequence we have
\begin{eqnarray*}
\f{d}{dt}L(W,Z)&\leq&-\f12D(W,Z)+C\e^2\\
&\leq&-2b L(W,Z)+C\|\varrho_0^{-1}u_0-\U(\mu;\cdot)\|^2e^{-2ah(t)},
\end{eqnarray*}
and the Gr\"onwall lemma gives
$$L(W,Z)\leq L(W_0,Z_0)e^{-2bt}+C\|\varrho_0^{-1}u_0-\U(\mu;\cdot)\|^2e^{-2bt}\int_0^te^{-2ah(s)+2bs}\,ds.$$
Since $\rho$ has been chosen to be less than $1=W_\infty,$ we can choose $b<a(1-\rho)$ in~\eqref{eq:entropy-entropy}
so that $b<a(1-\rho)<aW$ for all $(W,Z)$ in the stable neighbourhood ${\mathcal V}_\eta(W_\infty,Z_\infty).$
Then, since $\dot h=W$ and $h(0)=0,$ we have $-2ah(s)+2bs<-2\bigl(a(1-\rho)-b\bigr)s$ for all time $s>0,$ so there exists a constant $C>0$ such that
$$L(W,Z)\leq L(W_0,Z_0)e^{-2bt}+C\|\varrho_0^{-1}u_0-\U(\mu;\cdot)\|^2e^{-2bt}.$$
Now we use the equivalences~\eqref{eq:hopitalF} and \eqref{eq:hopitalG} to ensure the existence of $C>0$ such that
$$(W-1)^2+(\mu-f_p(Z))^2\leq C(\mu-f_p(Z_0))^2e^{-2bt}+C\|\varrho_0^{-1}u_0-\U(\mu;\cdot)\|^2e^{-2bt}.$$
Because $f'(I_\infty)\neq0,$ we can also find a constant $C>0$ such that
$$(W-1)^2+(Z-Z_\infty)^2\leq C(Z_0-Z_\infty)^2e^{-2bt}+C\|u_0-Z_\infty\U(\mu;\cdot)\|^2e^{-2bt}.$$
Using the Cauchy-Schwarz inequality as in \eqref{eq:Z-norm}, we find that
$$(Z_0-Z_\infty)^2\leq C\|u_0-Z_\infty\U(\mu;\cdot)\|^2,$$
so
$$(W-1)^2+(Z-Z_\infty)^2\leq C\|u_0-Z_\infty\U(\mu;\cdot)\|^2e^{-2bt}.$$
This inequality ensures that the terms $(i)$ and $(ii)$ decrease to zero exponentially fast.
It only remains to prove the same result for $(iii),$ and for this we use the explicit formula~\eqref{eq:explicitU}.
We obtain
\begin{align*}
\int_0^\infty \bigl(&\U(W^{-k}\mu^{-k}x)-\U(\mu^{-k}x)\bigr)^2(x+x^r)\,dx\\
&=C\int\left(e^{-\f{\beta}{\gamma\mu}(W^{-k}x)^\gamma}-e^{-\f{\beta}{\gamma\mu}x^\gamma}\right)^2(x+x^r)\,dx\\
&=C\int\left(e^{-\f{2\beta}{\gamma\mu}(W^{-k}x)^\gamma}+e^{-\f{2\beta}{\gamma\mu}x^\gamma}-2e^{-\f{\beta}{\gamma\mu}(1+W^{-1})x^\gamma}\right)(x+x^r)\,dx\\
&=C\int e^{-\f{2\beta}{\gamma\mu}y^\gamma}(W^ky+W^{rk}y^r)W^kdy+\int e^{-\f{2\beta}{\gamma\mu}x^\gamma}(x+x^r)\,dx\\
&\hspace{.4cm}-2\int e^{-\f{2\beta}{\gamma\mu}z^\gamma}\left(\left(\f{1+W^{-1}}{2}\right)^{-k}z+\left(\f{1+W^{-1}}{2}\right)^{-rk}z^r\right)\,\left(\f{1+W^{-1}}{2}\right)^{-k}dz\\
&=C\psi_1(W)\int e^{-\f{2\beta}{\gamma\mu}x^\gamma}x\,dx+C\psi_r(W)\int e^{-\f{2\beta}{\gamma\mu}x^\gamma}x^r\,dx
\end{align*}
where
$$\psi_r(W):=W^{(r+1)k}+1-2\left(\f{1+W^{-1}}{2}\right)^{-(r+1)k}.$$
Due to a Taylor expansion, we find that, locally,
$$|\psi_r(W)|\leq C(W-1)^2,$$
and so
$$\int_0^\infty \bigl(\U(W^{-k}\mu^{-k}x)-\U(\mu^{-k}x)\bigr)^2(x+x^r)\,dx\leq C\|u_0-Z_\infty\U(\mu;\cdot)\|^2e^{-2bt}.$$
Finally there exists a constant $C>0$ such that, for $\|u_0-Z_\infty\U(\mu;\cdot)\|$ small enough so that $(W,Z)$ stays in the neighbourhood ${\mathcal V}_\eta(W_\infty,Z_\infty),$ we have
\begin{eqnarray*}
\|u(t,\cdot)-Z_\infty\U(\mu;\cdot)\|&\leq& \|u(t,\cdot)-Q\U(W\mu;\cdot)\|+\|Q\U(W\mu;\cdot)-Z_\infty\U(\mu;\cdot)\|\\
&\leq& C\|u_0-Z_\infty\U(\mu;\cdot)\|e^{-bt}
\end{eqnarray*}
and we have shown the local exponential stability of the nontrivial steady states which satisfy $f'(I_\infty)<0$ in the case when $\kappa\equiv2.$

\

When $f'(I_\infty)>0,$ the steady state $(W_\infty,Z_\infty)$ is a saddle point of $L$ so it is unstable.
\end{proof}

\

We remark that the structure of the reduced system~\eqref{eq:ODEAB} is different for $p<1$ and $p>1.$
The nontrivial steady states are focuses in the case when $p<1$ and nodes for $p\geq1$ (see Figure~\ref{fig:convergence} for a numerical illustration in the case of Corollary~\ref{co:stability}).

\

\begin{figure}[h]

\begin{minipage}{0.45\textwidth}
\begin{center}
\includegraphics[width=.9\textwidth]{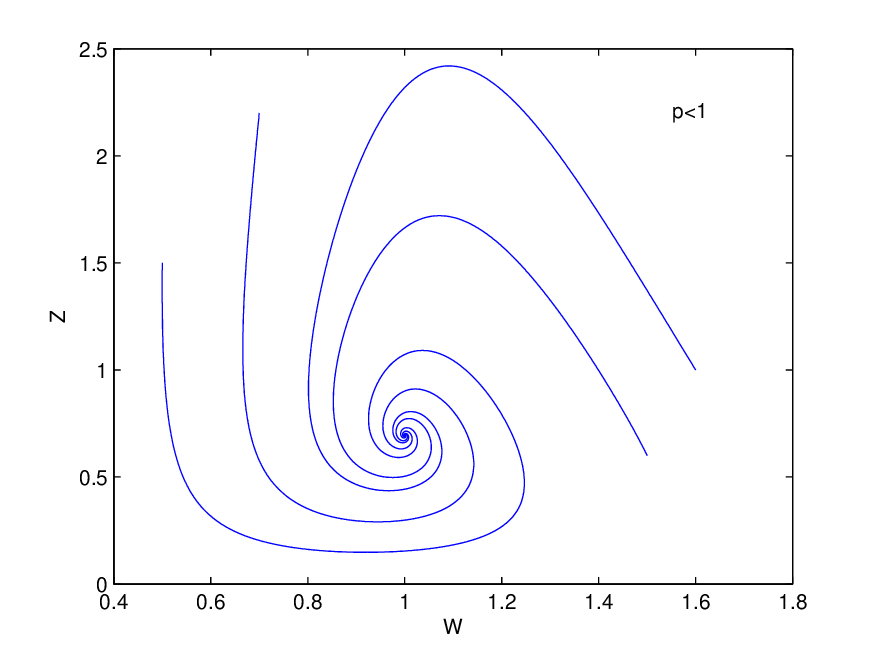}
\end{center}
\end{minipage}\hfill
\begin{minipage}{0.53\textwidth}
\begin{center}
\includegraphics[width=.9\textwidth]{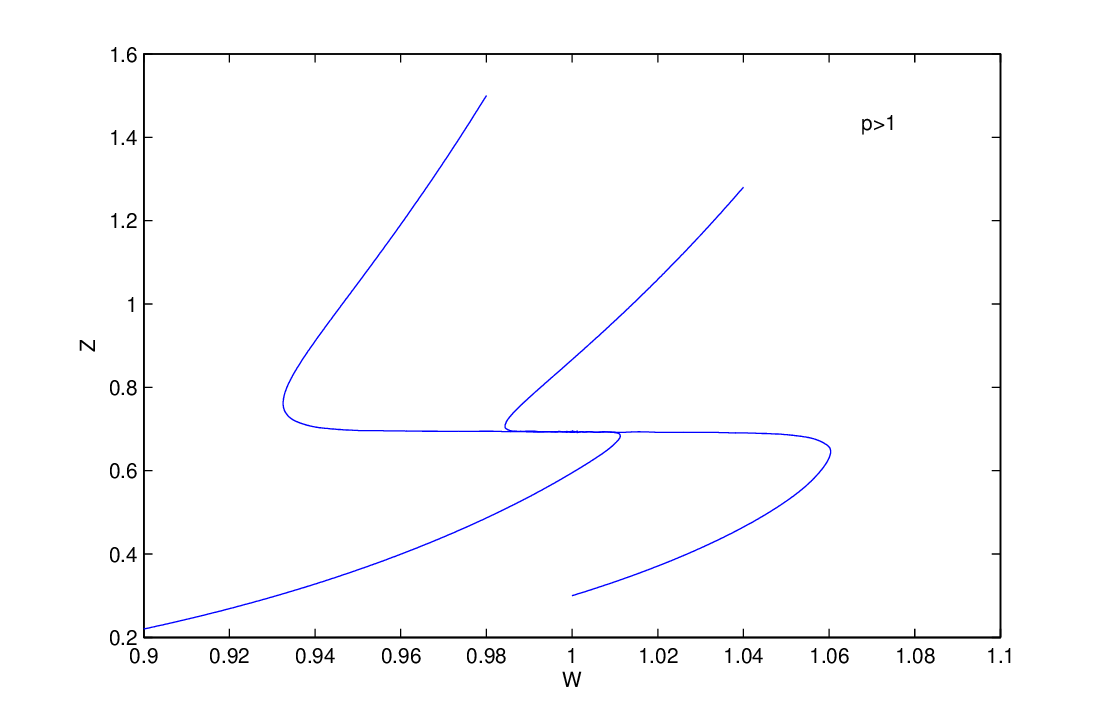}
\end{center}
\end{minipage}
\caption{Solutions to System~\eqref{eq:ODEAB} are plotted in the phase plane $(W,Z)$ for two different values of parameter $p.$
The other coefficients are $\gamma=0.1,$ $\mu=1$ and $f(x)=2e^{-x}.$
We can see that the steady state is a focus for $p<1$ (left) and a node for $p>1$ (right).}\label{fig:convergence}

\end{figure}

\

\section{Nonlinear Drift and Death Terms: Stable Persistent Oscillations}\label{sec:oscilld2}

We have seen in Theorem~\ref{th:nonlindrift:convergence} that any solution to the nonlinear equation~\eqref{eq:nonlinear} converges to a steady state.
Can this result be extended to Equation~\eqref{eq:nonlinear2} where the death rate is also nonlinear?
The result in Theorem~\ref{th:nonlindriftdeath} answers this question negatively.
Indeed it ensures the existence of functions $f$ and $g,$ and parameters $p$ and $q,$ such that Equation~\eqref{eq:nonlinear2} admits periodic solutions.
More precisely we prove, using the Poincar\'e-Bendixon theorem,
that any solution with an initial distribution in the eigenmanifold ${\mathcal E}$ which is not a steady state converges to a nontrivial periodic solution.
Then we extend this result by surrounding this set of initial distributions by an open neighbourhood in ${\mathcal H}.$

In the proof, we need to know the dependency of some quantities on the parameters $p$ and $q$.
Since we do not know the dependencies of $M_p=\int x^p\U(x)\,dx$ on $p,$ we consider an equation slightly different from \eqref{eq:nonlinear2}, namely
\beq\label{eq:nonlinear2bis}\f{\p}{\p t} u(t,x) = - f\left(\f{\int x^p u(t,x)}{\int x^p\U(x)}\right) \f{\p}{\p x} \bigl(xu(t,x)\bigr) - g\left(\f{\int x^q u(t,x)}{\int x^q\U(x)}\right) u(t,x) + {\mathcal F}_\gamma u(t,x).\eeq
Clearly the existence of functions $f$ and $g$ for which persistent oscillations appear in Equation~\eqref{eq:nonlinear2bis} ensures the same result for Equation~\eqref{eq:nonlinear2} (up to a dilation of $f$ and $g$).
Now let us make general assumptions on the two function $f$ and $g$ which allow one to obtain periodic oscillations.
Consider differentiable increasing functions $f$ and $g$ which satisfiy Assumption~\eqref{as:fandg} and define on $\R_+$ the function
\beq\label{def:psi}\psi(W):=f\left(W^{k(p-q)}g^{-1}(W)\right).\eeq
To ensure the existence and uniqueness of a nontrivial equilibrium, assume that
\beq\label{as:uniqueequilibrium}\exists !\,W_\infty\geq0,\quad \psi(W_\infty)=W_\infty\quad\text{and moreover}\quad \psi'(W_\infty)<1.\eeq
This steady state is unstable if, denoting $Q_\infty:=W_\infty^{-kq}g^{-1}(W_\infty),$ we have
\beq\label{as:instabilitycond}Q_\infty\left(pW_\infty^{kp}f'(W_\infty^{kp}Q_\infty)-W_\infty^{kq}g'(W_\infty^{kq}Q_\infty)\right)-\f{W_\infty}{k}>0.\eeq
Under these conditions, the solutions to Equation~\eqref{eq:nonlinear2bis} with an initial distribution close to the set ${\mathcal E}\setminus\{Q_\infty\U(W_\infty;\cdot)\}$ exhibit asymptotically periodic behaviors.
More precisely, we have the following result.

\medskip

\begin{theorem}
\label{th:PDEoscill}
Consider increasing differentiable functions $f$ and $g$ satisfying conditions~\eqref{as:fandg}, \eqref{as:uniqueequilibrium}, and \eqref{as:instabilitycond},
a parameter $\gamma\in(0,2],$ and a fragmentation kernel $\kappa$ which satisfies Assumption~\eqref{as:boundfragker}.
Then there exists an open neighbourhood ${\mathcal V}$ of ${\mathcal E}\setminus\{Q_\infty\U(W_\infty;\cdot)\}$ in ${\mathcal H}$ such that,
for any initial distribution $u_0\in{\mathcal V},$ there exist periodic functions $W(t)$ and $Q(t)$ such that
\beq\label{eq:convperiodic}\|u(t,\cdot)-Q(t)\U(W(t);\cdot)\|_{\mathcal H}\xrightarrow[t\to\infty]{}0.\eeq
\end{theorem}

\medskip

Before proving this theorem, we check that, if functions $f$ and $g$ satisfy either~\eqref{as:example1} or~\eqref{as:example2},
then Assumptions~\eqref{as:uniqueequilibrium} and \eqref{as:instabilitycond} are satisfied for well chosen parameters $p$ and $q.$
Thus Theorem~\ref{th:nonlindriftdeath} is a consequence of Theorem~\ref{th:PDEoscill}.

\

\noindent{\bf Example 1.} [Assumption~\eqref{as:example1}]
Assume that there exists $C>0$ such that for all $x\geq0,$ $g(x)\leq C\,xg'(x).$ Then $\psi(W)=W$ has a unique solution for $k(q-p)>C.$
Indeed, if we compute the derivative of $\psi$ we find
$$\psi'(W)=W^{k(p-q)}\left(k(p-q)\f{g^{-1}(W)}{W}+(g^{-1})'(W)\right)\,f'\left(W^{k(p-q)}g^{-1}(W)\right)$$
and $g(x)\leq C\,xg'(x)$ implies that $x(g^{-1})'(x)\leq C\,g^{-1}(x).$
So if $k(q-p)>C,$ $\psi$ decreases and Assumption~\eqref{as:uniqueequilibrium} is fulfilled.
If moreover $f(1)=g(1)=1,$ then the unique nontrivial equilibrium is given by $W_\infty=1.$
Then condition~\eqref{as:instabilitycond} is satisfied for $p>\f{g'(1)+\f1k}{f'(1)}.$

\

\noindent{\bf Example 2.} [Assumption~\eqref{as:example2}]
Consider the case $g(x)=x$ and $p=q,$ and assume that $f(x)-x$ has a unique root $x_0$ and $f'(x_0)-1<0.$
Then $\psi(W)=f(W)$ and Assumption~\eqref{as:uniqueequilibrium} is satisfied.
Moreover, condition~\eqref{as:instabilitycond} writes $pf'(W_\infty)>1+\f1k,$ so it is satisfied for $p$ large enough.

\

Now we give a lemma useful for the proof of Theorem~\ref{th:PDEoscill}.

\medskip

\begin{lemma}\label{lm:ODE}
Consider a dynamical system in $\R^n$ with a parameter $\e(t):$
\beq\label{eq:DySy}\dot X=F(X;\e(t)),\eeq
with $F\in{\mathcal C}(\R^n\times\R^n).$
Assume that for any vanishing parameter $\|\e(t)\|\xrightarrow{t\to\infty}0$ the solutions to Equation~\eqref{eq:DySy} are bounded.
Then for any solution $X^\e$ associated to $\|\e(t)\|\xrightarrow{t\to\infty}0,$ there exists a solution $X^0$ associated to $\e\equiv0$ such that $X^\e$ and $X^0$ have the same $\omega-$limit set.
\end{lemma}

\

\begin{proof}[Proof of Lemma~\ref{lm:ODE}]

Let $X(t)$ be a solution to System~\eqref{eq:DySy} with $\|\e(t)\|\to0.$
By assumption, $X(t)$ is bounded, so $\dot X(t)$ is also bounded since $F$ is continuous.
Now consider a sequence $\{t_k\}_{k\in\N}$ which tends to infinity and define the sequence $\{X_k(\cdot)\}$ by $X_k(t)=X(t+t_k).$
This sequence is bounded in $W^{1,\infty}(\R_+),$ so there exists a subsequence which converges to $X_\infty(\cdot).$
This limit is a solution to Equation~\eqref{eq:DySy} with $\e\equiv0.$
We take $X^0:=X_\infty,$ which ends the proof of Lemma~\ref{lm:ODE}.
\end{proof}

\

\begin{proof}[Proof of Theorem~\ref{th:PDEoscill}]
We divide the proof in two parts: first the result for $u_0\in{\mathcal E}$ and then the existence of a neighbourhood ${\mathcal V}$ of ${\mathcal E}$ in ${\mathcal H}$ where the result persists.

\

\noindent{\bf First step: $u_0\in{\mathcal E}.$}\\
For $u_0\in{\mathcal E}\setminus\{0\},$ there are $W_0>0$ and $Q_0>0$ such that
$$u_0(x)=Q_0\U(W_0;x).$$
Then, if $u(t,x)$ is the solution to Equation~\eqref{eq:nonlinear2bis} and $W$ is the solution to
$$\dot W=\f Wk\left(f\left(\f{M_p[u](t)}{M_p[\U]}\right)-W\right)$$
with $W(0)=W_0,$ the relation holds for all $t>0$ and $x>0:$
$$u(t,x)=Q(t)\U(W(t);x),$$
where $Q(t):=Q_0 e^{\int_0^t(W(s)-g(M_q[u](s)/M_q[\U]))\,ds}.$
Then we can compute
$$M_p[u](t)=\int_0^\infty x^pu(t,x)\,dx=W^{kp}(t)Q(t)\,M_p[\U],$$
and finally we obtain the reduced system of ODEs satisfied by $(W,Q):$
\begin{equation}\label{eq:nonlinODE2}
\left\{\begin{array}{rcl}
\dot W&=&\dis\f W k\left(f\left(W^{kp}Q\right)-W\right),
\vspace{.2cm}\\
\dot Q&=&\dis Q\left(W-g\left(W^{kq}Q\right)\right).
\end{array}\right.
\end{equation}
We prove that System~\eqref{eq:nonlinODE2} has bounded solutions and a unique positive steady state which is unstable.
Then we use the Poincar\'e-Bendixon theorem to ensure the convergence to a limit cycle.

The fact that $0<f(0)\leq f\leq f(\infty)<\infty$ and that $g$ increases from $0$ to the $\infty$ ensures that the solution remains bounded.
Let $(W_\infty,Q_\infty)$ be a positive steady state.
It satisfies
$$W_\infty=f\left(W^{kp}Q\right)=g\left(W^{kq}Q\right)$$
and so, since $g$ is invertible, $Q_\infty=W_\infty^{-kq}g^{-1}(W_\infty).$
Then $W_\infty$ is solution to the equation
$$W_\infty=f\left(W_\infty^{k(p-q)}g^{-1}(W_\infty)\right)=\psi(W_\infty),$$
and Assumption~\eqref{as:uniqueequilibrium} ensures the uniqueness of such a solution.
Now look at the stability of this positive steady state.
We write system~\eqref{eq:nonlinODE2} in the form
$$\left(\begin{array}{c}
\dot W\\
\dot Q
\end{array}\right)
=F\left(\begin{array}{c}
W\\
Q
\end{array}\right),$$
so we have
$$Jac(F)_{eq}=\left(\begin{array}{cc}
pW_\infty^{kp}Q_\infty f'(W_\infty^{kp}Q_\infty)-\f1k W_\infty&\f{1}{k}W_\infty^{kp+1}f'(W_\infty^{kp}Q_\infty)\\
Q_\infty-kqW_\infty^{kq-1}Q_\infty^2 g'(W_\infty^{kq}Q_\infty)&-W_\infty^{kq}Q_\infty g'(W_\infty^{kq}Q_\infty)
\end{array}\right).$$
The trace of this matrix is
$$T=Q_\infty\left(pW_\infty^{kp}f'(W_\infty^{kp}Q_\infty)-W_\infty^{kq}g'(W_\infty^{kq}Q_\infty)\right)-\f{W_\infty}{k}$$
and the determinant is
$$D=\f Wk Q\left( W^{kq}g'(W^{kq}Q)- W^{kp}f'(W^{kp}Q)\right)+(q-p)W^{k(p+q)}Q^2f'(W^{kp}Q)g'(W^{kq}Q).$$
We know from Assumption~\eqref{as:uniqueequilibrium} that $\psi'(W_\infty)<1$ and, if we compute $\psi'(W),$ we find
$$\psi'(W)=\left[k(p-q)W^{k(p-q)-1}g^{-1}(W)+\f{W^{k(p-q)}}{g'(g^{-1}(W))}\right]f'(W^{k(p-q)}g^{-1}(W)).$$
Since $g^{-1}(W)=W^{kq}Q$ we finally obtain
$$D=\f1k W_\infty^{kq+1}Q_\infty\, g'(W_\infty^{kq}Q_\infty)(1-\psi'(W_\infty))>0.$$
Thus when $T>0,$ namely when Assumption~\eqref{as:instabilitycond} is satisfied, the two eigenvalues have positive real parts and the positive steady state is unstable.
Now we prove that $(W,Q)$ remains away from the boundaries of $(\R_+)^2.$
For this we write that
$$\forall t>0,\quad \underline W:=\min(W_0,f(0))\leq W(t)\leq\max(W_0,f(\infty)):=\overline W,$$
and then
$$Q\geq\min(Q_0,{\overline W}^{-kp}g^{-1}(\underline W)).$$
Since $f(0)>0,$ $\underline W>0$ for $W_0>0$ so any solution with $W_0>0$ and $Q_0>0$ stays a positive distance from the boundaries of $(\R_+)^2.$
Then the Poincar\'e-Bendixon theorem (see~\cite{hofbauer} for instance) ensures that any solution to System~\eqref{eq:nonlinODE2} with $W_0>0,$ $Q_0>0,$ and $(W_0,Q_0)\neq(W_\infty,Q_\infty)$ converges to a limit cycle.

\

\noindent{\bf Second step: Existence of ${\mathcal V}.$}\\
Let $u_0\not\equiv0$ in ${\mathcal H}$ and build from $u(t,x),$ a solution to Equation~\eqref{eq:nonlinear2bis} with initial distribution $u_0,$ a function $v$ by
$$v(h(t),x)=W^k(t)u(t,W^k(t)x)e^{\int_0^t(g(M_q[u](s)/M_q[\U])-W(s))\,ds},$$
with $W$ a solution to
$$\dot W=\f Wk\left(f\left(\f{M_p[u]}{M_p[\U]}\right)-W\right),$$
and $h$ solution to $\dot h=W$ with $h(0)=0.$
We have already seen in Section~\ref{ssec:GRE} that $h:\R_+\to\R_+$ is one to one since $h(t)\geq k\ln(1+\f tk).$
We take $W(0)=1$ to have $v(t=0,\cdot)=u(t=0,\cdot)=u_0.$
Due to Theorem~\ref{th:nu1} we know that $v$ is a solution to
$$\p_tv(t,x)+\p_x\bigl(x\,v(t,x)\bigr)+v(t,x)={\mathcal F}_\gamma v(t,x),$$
and the GRE ensures the convergence
$$v(t,x)\xrightarrow[t\to\infty]{}\left(\int\phi(x)u_0(x)\,dx\right)\U(x).$$
As a consequence we have the equivalences, for any $p\geq0,$
$$M_p[u](t)\underset{t\to\infty}{\sim}\varrho_0 M_p[\U]\,W^{kp} e^{\int_0^t(W(s)-g(M_q[u](s)/M_q[\U]))\,ds}$$
so, if we define $Q(t):=\varrho_0 e^{\int_0^t(W(s)-g(M_q[u](s)/M_q[\U])\,ds},$ we find that the reduced system~\eqref{eq:nonlinODE2} is ``asymptotically equivalent'' to Equation~\eqref{eq:nonlinear2bis}.
More precisely, defining
$$\e_p(t)=\f{M_p[u](t)}{M_p[\U]W^{kp}(t)Q(t)}-1$$
as in the proof of Theorem~\ref{th:convergence}, we have that $\e_p\to0$ and $(W,Q)$ is solution to
\begin{equation}\label{eq:nonlinODE2e}
\left\{\begin{array}{rcl}
\dot W&=&\dis\f W k\left(f\left((1+\e_p)W^{kp}Q\right)-W\right),
\vspace{.2cm}\\
\dot Q&=&\dis Q\left(W-g\left((1+\e_q)W^{kq}Q\right)\right).
\end{array}\right.
\end{equation}
Now we prove that if $W_0$ and $Q_0$ are positive and $(W_0,Q_0)\neq(W_\infty,Q_\infty),$ then if $\|u_0-Q_0\U(W_0;\cdot)\|_{\mathcal H}$ is small enough,
the solution $u$ to Equation~\eqref{eq:nonlinear2bis} converges to a periodic solution.
Denote by $d$ the distance between $(W_0,Q_0)$ and $(W_\infty,Q_\infty).$
Since $(W_\infty,Q_\infty)$ is a source for System~\eqref{eq:nonlinODE2}, there exists a ball with radius $\rho<d$ such that the flux is outgoing, namely
\beq\label{eq:outgoflux}\forall\,(W,Q)\in\p B((W_\infty,Q_\infty),\rho),\quad F(W,Q)\cdot {\bf n}>0,\eeq
where ${\bf n}$ is the outgoing normal of $B((W_\infty,Q_\infty),\rho).$
Then, if we define by $F(W,Q;\e_p,\e_q)$ the flux of Equation~\eqref{eq:nonlinODE2e},
we have by continuity of $f$ and $g$ that there exists $\e_0$ such that~\eqref{eq:outgoflux} remains true for $F(W,Q;\e_p,\e_q)$ provided that $\e_p$ and $\e_q$ stay less than $\e_0.$
But we know from the proof of Theorem~\ref{th:convergence} that there exists a constant $C_p>0$ such that for all time $t>0,$ $\e_p\leq C_p\|u_0-Q_0\U(W_0;\cdot)\|.$
So for $\|u_0-Q_0\U(W_0;\cdot)\|\leq\f{\e_0}{C_p+C_q},$ the solution to System~\eqref{eq:nonlinODE2e} cannot converge to the positive steady state $(W_\infty,Q_\infty).$
Thanks to the same arguments, if $\|u_0-Q_0\U(W_0;\cdot)\|$ is small enough, then $(W,Q)$ remains away from the boundaries of $(\R_+)^2.$
We obtain due to Lemma~\ref{lm:ODE} that for $\|u_0-Q_0\U(W_0;\cdot)\|$ small enough, $(W(t),Q(t))$ converges to a limit cycle $(\widetilde W(t),\widetilde Q(t)).$
Then we write
$$\|u-\widetilde Q\U(\widetilde W;\cdot)\|\leq\|u-Q\U(W;\cdot)\|+\|Q\U(W;\cdot)-\widetilde Q\U(\widetilde W;\cdot)\|$$
and we conclude as in the proof of Theorem~\ref{th:convergence} that the solution $u$ to Equation~\ref{eq:nonlinear2} converges in ${\mathcal H}$ to $\widetilde Q\,\U(\widetilde W;\cdot).$
Finally we have proved, for any $(W_0,Q_0)\in(\R_+^*)^2\setminus\{(W_\infty,Q_\infty)\},$
the existence of a ball centered in $Q_0\U(W_0;\cdot)$ such that any solution to Equation~\eqref{eq:nonlinear2bis} with an initial distribution in this ball converges to a periodic solution.
Then Theorem~\ref{th:PDEoscill} is proved for ${\mathcal V}$ the union of all these balls.

\end{proof}

\bigskip

To illustrate the convergence to a periodic solution for solutions to Equation~\eqref{eq:nonlinear2bis},
we plot in Figure~\ref{fig:oscillations} a solution to Equation~\eqref{eq:nonlinODE2} with an initial distribution close to the steady state $(W_\infty,Q_\infty)$
and for coefficients which satisfy the assumptions of Theorem~\ref{th:PDEoscill}.

\

\begin{figure}[h]

\begin{minipage}{0.49\textwidth}
\begin{center}
\includegraphics[width=.9\textwidth]{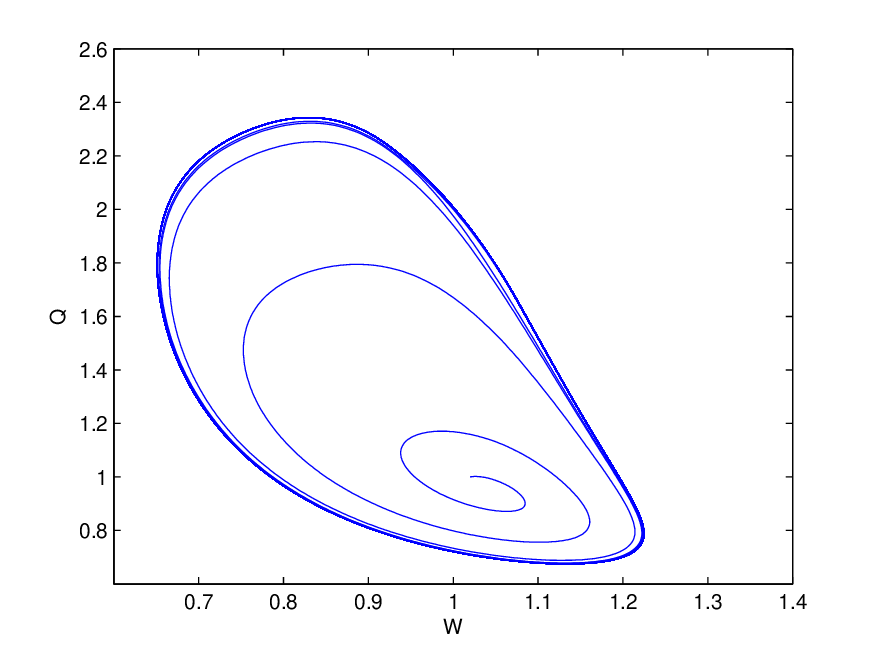}
\end{center}
\end{minipage}\hfill
\begin{minipage}{0.49\textwidth}
\begin{center}
\includegraphics[width=.9\textwidth]{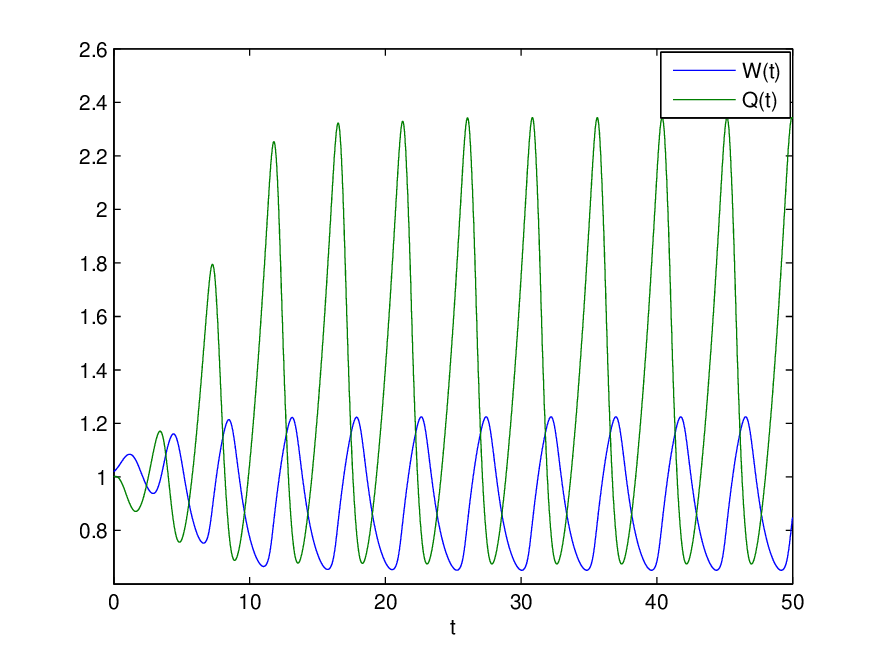}
\end{center}
\end{minipage}
\caption{A solution to System~\eqref{eq:nonlinODE2} is plotted in the phase plane $(W,Q)$ (left) and as a function of the time (right).
The coefficients are $\gamma=1,$ $p=2,$ $q=5,$ $f(x)=1+e^{-1}-e^{-x^4},$ and $g(x)=0.9 x.$}\label{fig:oscillations}

\end{figure}

\

\section{The Prion Equation: Existence of Periodic Solutions}\label{sec:prion}

Prion diseases are believed to be due to self-replication of a pathogenic protein through a polymerization process not yet very well understood (see \cite{Lenuzza} for more details).
To investigate the replication process of this protein, a mathematical PDE model was introduced by \cite{Greer}.
We recall this model under a form slightly different from the original one (see \cite{CL1,DGL} for the motivations to consider this form):
\begin{equation}
\label{eq:prion}
\left\{
\begin{array}{rll}
\dfrac{dV(t)}{dt}&=&\displaystyle \lambda- V(t)\left[\delta+ \int_{0}^{\infty} \tau(x) u(t,x) \; dx\right],
\vspace{.2cm}\\
\dfrac{\partial}{\partial t} u(t,x) &=& \displaystyle - V(t) \frac{\partial}{\partial x} \big(\tau(x) u(t,x)\big) - \mu(x) u(t,x) + {\mathcal F}u(t,x).
\end{array} \right.
\end{equation}
In this equation, $u(t,x)$ represents the quantity of polymers of pathogenic proteins of size $x$ at time $t,$ and $V(t)$ the quantity of normal proteins (also called monomers).
The polymers lengthen by attaching monomers with the rate $\tau(x),$ die with the rate $\mu(x),$ and split into smaller polymers with respect to the fragmentation operator ${\mathcal F}.$
The quantity of monomers is driven by an ODE with a death parameter $\delta$ and production rate $\lambda.$
This ODE is quadratically coupled to the growth-fragmentation equation because of the polymerization mechanism, which is assumed to follow the mass action law.

This system admits a trivial steady state, also called disease-free equilibrium since it corresponds to a situation where no pathogenic polymer is present:
\mbox{$V=\f\lambda\delta$} and \mbox{$u\equiv0.$}
The stability of this steady state has been investigated in \cite{CL2,CL1,SimonettWalker,Walker} under general assumptions on the coefficients.
It depends on the sign of the principal eigenvalue of the linear growth-fragmentation with a frozen transport term $V=\f\lambda\delta.$
The existence of nontrivial steady states (also called endemic equilibria) has also been investigated, and it is proved in~\cite{CDG} that several can exist.
But the stability (even linear) of these nontrivial steady states is a difficult and still open problem for general coefficients.
The only existing results concern the ``constant case'' ($\tau$ constant, $\beta$ linear and $\kappa$ constant) initially considered by \cite{Greer},
since then the model reduces to a closed system of ODEs.
In this case, the problem has been entirely solved by \cite{Engler,Greer,Pruss}:
the disease-free steady state is globally stable when it is the only equilibrium, and, when an endemic equilibrium exists, this endemic equilibrium is unique and globally stable.

A new, more general model has been introduced in \cite{Greer2} and takes into account the incidence of the {\it total mass} of polymers $P(t):=\int xu(t,x)\,dx$ on the polymerization process.
More precisely, they consider that the presence of many polymers reduces the attaching process of monomers to polymers
by multiplying the polymerization rate by $\f{1}{1+\omega P(t)}$ with $\omega$ a positive parameter.
Then they prove similar results about the existence and stability of steady states, still in the case of constant parameters.

Here we look at a generalization of the influence of polymers on the polymerization rate by considering the system
\begin{equation}
\label{eq:Webbgen}
\left\{
\begin{array}{rll}
\dfrac{dV(t)}{dt}&=&\displaystyle - V(t) f\left(\int x^pu\right)\int_{0}^{\infty} \tau(x) u(t,x) \; dx - \delta V(t) + \lambda\,,
\vspace{.2cm}\\
\dfrac{\partial}{\partial t} u(t,x) &=& \displaystyle - V(t)f\left(\int x^pu\right) \frac{\partial}{\partial x} \big(\tau(x) u(t,x)\big) - \mu\, u(t,x) + {\mathcal F} u(t,x),
\end{array} \right.
\end{equation}
where $p\geq0$ and $f:\R_+\to\R_+$ is a differentiable function.
In this framework, the model of \cite{Greer2} corresponds to $p=1$ and $f(P)=\f{1}{1+\omega P},$ together with $\tau$ constant, $\beta$ linear, and $\kappa$ constant.
Using the reduction method to ODEs, we prove that such a system can exhibit periodic solutions.
For this we consider the following system, where $M_p=M_p[\U],$
\begin{equation}
\label{eq:prionoscill}
\left\{
\begin{array}{rll}
\dfrac{dV(t)}{dt}&=&\displaystyle - V(t) f\left(\mu^{-kp}\f{M_1}{M_p}\int x^pu\right)\int_{0}^{\infty} x\, u(t,x) \; dx - \delta V(t) + \lambda\,,
\vspace{.2cm}\\
\dfrac{\partial}{\partial t} u(t,x) &=& \displaystyle - V(t)f\left(\mu^{-kp}\f{M_1}{M_p}\int x^pu\right) \frac{\partial}{\partial x} \big(x\, u(t,x)\big) - \mu\, u(t,x) + {\mathcal F}_\gamma u(t,x),
\end{array} \right.
\end{equation}
which is a particular case of System~\eqref{eq:Webbgen}, with coefficients satisfying the assumptions of Theorem~\ref{th:nu1} and up to a dilation of $f.$
We prove that, under Assumption~\eqref{as:fandgprion} on the incidence function $f,$ there exist values of the coefficients for which System~\eqref{eq:prionoscill} admits nontrivial periodic solutions.
This result is stated in the following theorem, which is a more detailed version of Theorem~\ref{th:prion}.

\medskip

\begin{theorem}\label{th:oscillprion}
Under the assumptions of Theorem~\ref{th:prion}, there exists $p>0$ for which Equation~\eqref{eq:prionoscill} admits a solution of the form
$$\bigl(V(t),\,u(t,x)=Q(t)\,\U(W(t);x)\bigr),$$
with $V,$ $W,$ and $Q$ nontrivial periodic functions.
\end{theorem}

\

\begin{proof}

\noindent{\bf First step: reduced dynamic in ${\mathcal E}$}\\
We look at the dynamic of System~\eqref{eq:prionoscill} on the invariant eigenmanifold ${\mathcal E}.$
For any initial condition $u_0$ in ${\mathcal E},$ there exist $Q_0$ and $W_0$ such that $u_0$ writes as
$$u_0(x)=\f{1}{M_1}Q_0\,\U(W_0;x).$$
Consider the solution to System~\eqref{eq:prionoscill} corresponding to this initial data and define $W$ as the solution to
\begin{equation}\label{eq:systODEprion1}
\left\{\begin{array}{cll}
\dot W&=&\dis\f Wk\left(f\left(\mu^{-kp}\f{M_1}{M_p}\int x^pu\right)V- W\right),
\vspace{.2cm}\\
\dot W(0)&=&W_0.
\end{array}\right.
\end{equation}
Then we know from Theorem~\ref{th:nu1} that the solution satisfies
$$u(t,x)=\f{Q_0}{M_1}\U(W(t);x)e^{\int_0^t(W(s)-\mu)\,ds},$$
which allows one to compute
$$\int_0^\infty x^pu(t,x)\,dx=Q_0\f{M_p}{M_1}W^{kp}e^{\int_0^t(W(s)-\mu)\,ds}$$
and
$$\int_0^\infty x\,u(t,x)\,dx=Q_0W^k\,e^{\int_0^t(W(s)-\mu)\,ds}.$$
Thus, defining $Q(t):=Q_0\,e^{\int_0^t(W(s)-\mu)\,ds},$ Equation~\eqref{eq:systODEprion1} becomes
$$\dot W=\f Wk\left(f\left((\mu^{-1}W)^{kp}Q\right)V-W\right),$$
and System~\eqref{eq:prionoscill} reduces to
\begin{equation}\label{eq:systODE4}
\left\{\begin{array}{rcl}
\dot V&=&\dis\lambda-V\left(\delta+f\left((\mu^{-1}W)^{kp}Q\right)W^k Q\right),
\vspace{.2cm}\\
\dot W&=&\dis\f Wk\left(f\left((\mu^{-1}W)^{kp}Q\right)V-W\right),
\vspace{.2cm}\\
\dot Q&=&\dis Q\left(W-\mu\right).
\end{array}\right.
\end{equation}
Now we prove that System~\eqref{eq:systODE4} admits a unique nontrivial steady state which undergoes a supercritical Hopf bifurcation when $p$ increases from $0.$ 

\

\noindent{\bf Second step: Hopf bifurcation for the reduced system}\\
First we look for a positive steady state of System~\eqref{eq:systODE4}.
Such a steady state is unique and given by
$$W_\infty=\mu,$$
$$V_\infty=\f{1}{\delta}\left(\lambda-\mu^{k+1}Q_\infty\right),$$
where $Q_\infty$ satisfies
$$f\left(Q_\infty\right)=\f{\delta\mu}{\lambda-\mu^{k+1}Q_\infty}=:g(Q_\infty).$$
Such a $Q_\infty$ exists and is unique by Assumption~\eqref{as:fandgprion}, and moreover it satisfies $\lambda-\mu^{k+1}Q_\infty>0$ and so $V_\infty$ is positive.
Finally, there exists a unique positive steady state.
Now the method consists in considering the power $p$ as a bifurcation parameter and to prove that the unique positive steady state undergoes a supercritical Hopf bifurcation when $p$ increases.
The linear stability of the steady state is given by the eigenvalues of the Jacobian matrix
$$Jac_{eq}=\left(\begin{array}{ccc}
-\delta-\mu^kQf(Q) & -k\mu^kVQ(f(Q)+pQf'(Q)) & -\mu^kV(f(Q)+Qf'(Q))\\
\f{\mu}{k} f(Q) & p\mu VQf'(Q)-\f{\mu}{k} & \f{\mu}{k}Vf'(Q)\\
0 & Q & 0
\end{array}\right),$$
where the $_\infty$ indices are suppressed for the sake of clarity.
The trace of this matrix is
$$T=-\delta - \mu^kQf(Q)-\f\mu k+p\mu VQf'(Q),$$
which is negative for $p<p_1$ and positive for $p>p_1$ with
$$p_1:=\f{\delta+\mu/k+\mu^kQf(Q)}{\mu VQf'(Q)}>0.$$
The determinant is
$$D=\f{\mu}{k}VQ\left(\delta f'(Q)-\mu^kf^2(Q)\right).$$
It is independent of $p$ and negative since $f'(x_0)<g'(x_0)$ and
$$g'(x_0)=\f{\delta\mu^{k+2}}{(\lambda-\mu^{k+1}x_0)^2}=\f{\mu^k}{\delta}f^2(x_0).$$
The sum of the three $2\times2$ principal minors is
$$M=-\delta pVQf'(Q)+\f{\mu\delta}{k}+\mu^k\Bigl(\mu+\f1k\Bigr) Qf(Q)-\f\mu k VQf'(Q).$$
To use the Routh-Hurwitz criterion, let define $\psi(p):=MT-D$ and look at its sign.
For $p=0$ we have
\begin{align*}\psi(0)=-\f\mu k\biggl[\delta^2+\delta Qf(Q)+&\f{\mu\delta}{k}+\left(\delta(k+\mu^{-1})-\mu\right)\mu^kQf(Q)\\
&+VQ\left(\mu^{k-1}(k+\mu^{-1})f^2(Q)-f'(Q)\right)\left(\mu^kQf(Q)+\f\mu k\right)\biggr],
\end{align*}
and it is negative since $\mu\leq(k+\mu^{-1})\delta$ and $f'(Q)<g'(Q)=\f{\mu^k}\delta f^2(Q)\leq\mu^{k-1}(k+\mu^{-1})f^2(Q).$
For $p=p_1,$ it is positive because $\psi(p_1)=-D>0.$
Now we investigate the variations of $\psi$ between $0$ and $p_1.$
The first derivative of $T,$ $D,$ and $M$ are given by
$$T'(p)=\mu VQf'(Q),\quad M'(p)=-\delta VQf'(Q),\quad D'(p)=0,$$
and the second derivatives are all null:
$$T''(p)=M''(p)=D''(p)=0.$$
So we have
$$\psi''(p)=2M'(p)T'(p)<0$$
and $\psi$ is concave.
Thus there exists a unique $p_0\in(0,p_1)$ such that $\psi(p_0)=0.$
Now we can use the Routh-Hurwitz criterion (see \cite{hofbauer} for instance).
For $0\leq p<p_0$ we have $T<0,$ $D<0,$ and $MT<D,$ so the steady state is linearly stable with one real negative eigenvalue and two complex conjugate eigenvalues with a negative real part.
For $p_0<p<p_1$ we have $T<0,$ $D<0,$ and $MT>D,$ so the steady state is linearly unstable with one real negative eigenvalue and two complex conjugate eigenvalues with a positive real part.
The two conjugate eigenvalues cross the imaginary axis when $p=p_0$ so there is a Hopf bifurcation at this point.
To prove that a periodic solution appears with this bifurcation, it remains to check that the complex eigenvalues cross the imaginary axis with a positive speed (see \cite{Francoise} for instance).
Denote by $a\pm ib$ the two conjugate eigenvalues and $c<0$ the real one.
We have to prove that the derivative $a'(p_0)>0.$ 
For this we express $\psi(p)$ in terms of $a(p),$ $b(p),$ and $c(p),$ and we use the concavity of $\psi.$
We have for any $p$
$$T=2a+c,\quad D=c(a^2+b^2),\quad M=a^2+b^2+2ac,$$
so
$$\psi(p)=2a(a^2+b^2)+4a^2c+2ac^2.$$
Then, using that $a(p_0)=0$ by the definition of $p_0,$ we obtain
$$\psi'(p_0)=2(b^2+c^2)a'.$$
But $\psi'(p_0)>0$ because $\psi$ is concave and increasing on a neighbourhood of $p_0,$ so necessarily $a'(p_0)>0.$
This proves the existence of a periodic solution $(V,W,Q)$ to System~\eqref{eq:systODE4} for a parameter $p\geq p_0$ close to $p_0.$
Then the functions $V(t)$ and $Q(t)\,\U(W(t),x)$ are periodic and solve System~\eqref{eq:prionoscill}.
\end{proof}

\

To know whether such a periodic solution is stable is difficult, even for the reduced dynamics~\eqref{eq:systODE4}.
Nevertheless we give in Figure~\ref{fig:oscillations_prion} evidence that it should be the case.
This simulation is made with parameters and a function $f$ satisfying Assumption~\eqref{as:fandgprion}, for a value of parameter $p>p_0.$
It seems to indicate that the periodic solution persists for $p$ away from $p_0.$

\

\begin{figure}[h]

\begin{minipage}{0.49\textwidth}
\includegraphics[width=\textwidth]{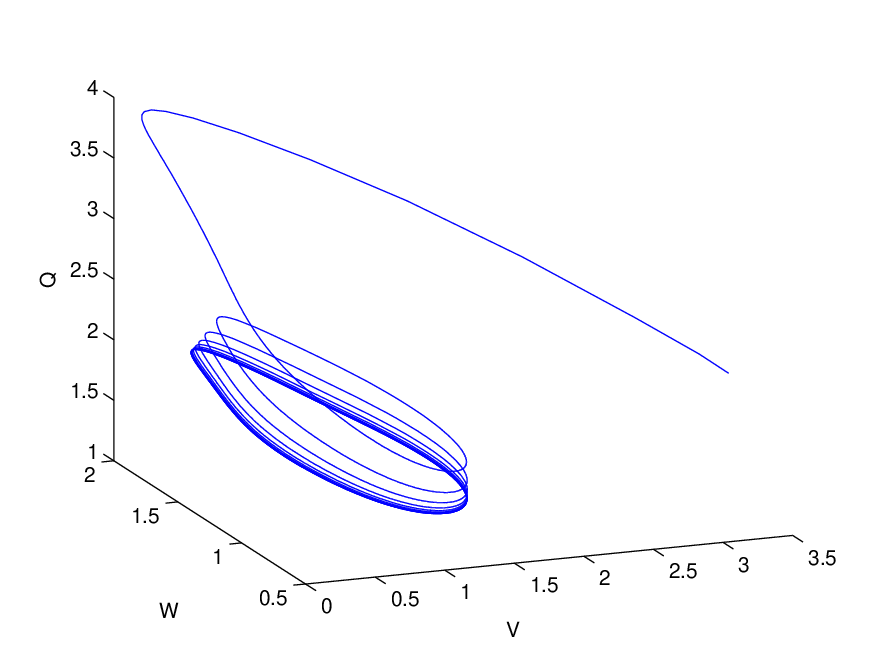}
\end{minipage}\hfill
\begin{minipage}{0.49\textwidth}
\includegraphics[width=\textwidth]{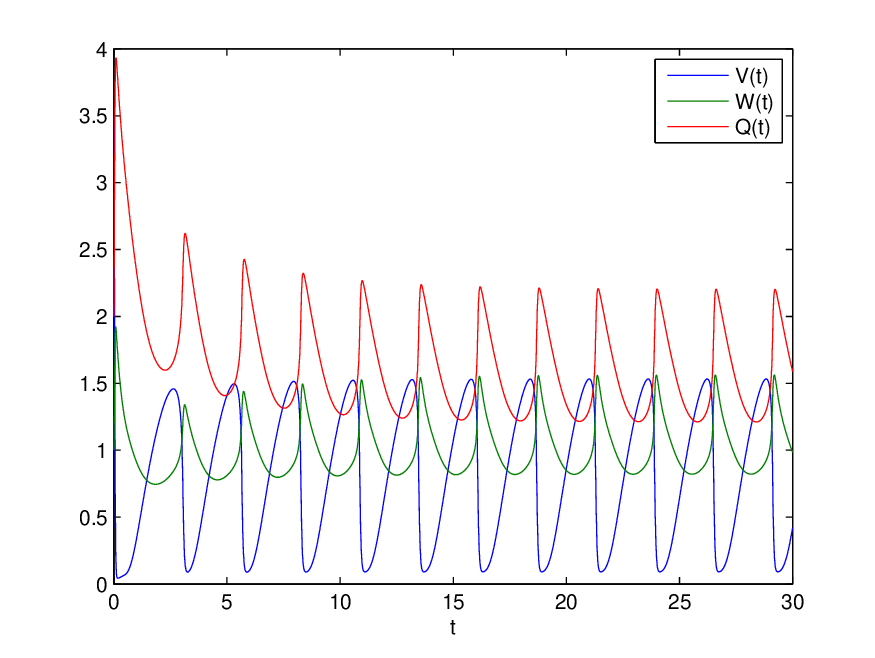}
\end{minipage}

\

\caption{A solution to System~\eqref{eq:systODE4} is plotted in the phase plane $(W,Q)$ (left) and as a function of the time (right).
The coefficients are $\lambda=0.9,$ $\delta=0.2,$ $\mu=\gamma=1,$ $f(x)=6.3(1.1-e^{\f{-x^2}{20}}),$ and $p=4.$}\label{fig:oscillations_prion}

\end{figure}

\

\section{Comparison between Perron and Floquet Eigenvalues}\label{sec:Floquet}

In this Section, we assume that the time-dependent terms $V(t)$ and $R(t)$ of the growth-fragmentation equation are $T$-periodic controls.
Periodic controls are usually used in structured equations to model optimization problems.
In the case of prion diseases (see Section~\ref{sec:prion}), there exists an amplification protocol called PMCA
(Protein Misfolded Cyclic Amplification; see \cite{Lenuzza} and references therein for more details)
which consists of periodically {\it sonicating} a sample of prion polymers in order the break them into smaller ones and thus increase their quantity.
Between these phases of sonication, the sample is flooded with a large quantity of monomers in order to allow a fast polymerization process.
This protocol can be modeled by introducing in the growth-fragmentation equation a periodic control in front of the fragmentation operator \cite{CDG,CalvezGabriel,GabrielPhD}.
Then a problem is to find a periodic control which maximizes the proliferation rate of the polymers in the sample.
Mathematically this leads to the problem of optimizing the Floquet eigenvalue of the growth-fragmentation equation, namely the eigenvalue associated to periodic coefficients (see~\cite{BP} for instance).
Before solving this difficult question, a first step is to compare the Floquet eigenvalue to the Perron eigenvalue associated to constant coefficients,
for instance the mean value of the periodic control, and to know whether the Floquet one can be better than the Perron one.
Such concerns are also investigated in the context of circadian rhythms for the optimization of chronotherapy (see \cite{Lepoutre,Lepoutre2,Clairambault}).
The population is an age structured population of cells and the model is a system of renewal equations.
The death and birth rates are assumed to be periodic, and the Floquet eigenvalue is compared to the Perron eigenvalue associated to geometrical or arithmetical time average of the periodic coefficients.
Comparison of results obtained show that the Floquet eigenvalue can be greater or less than the Perron one depending on parameters.

Here the controls are on the growth and death coefficients, and we compare results between Floquet and Perron eigenvalues in the case where $\nu=1$ or $\nu=0$ and $\gamma=1.$
The Floquet eigenelements $(\Lambda_F,\U_F)$ associated to periodic controls are defined by two properties:
$\U_F(t,x)e^{\Lambda_Ft}$ is a solution to Equation~\eqref{eq:Floquet} and $\U_F(t,x)$ is a $T$-periodic function of the time.
For any $T$-periodic function $P(t),$ we use the notation
$$\overline P:=\f1T\int_0^T P(t)\,dt.$$
To ensure the uniqueness of Floquet eigenfunction, we impose $\overline{\int_0^\infty \U_F(t,x)\,dx}=1.$
Then we have the following comparison results.

\medskip

\begin{proposition}
Assume that conditions~\eqref{as:powerlaw}-\eqref{as:selfsimfrag} are satisfied with $\nu=1,$ and define
from $V(t)$ the $T$-periodic control on growth, $W(t),$ as the periodic solution to
\beq\label{eq:ODEWF}\dot W=\f {\tau W}k(V-W).\eeq
Then
\beq\label{eq:FloquetU}\exists C>0,\qquad\U_F(t,x)=C\,\U(W(t);x)\,e^{\int_0^t(\Lambda(W(s),R(s))-\Lambda_F)\,ds},\eeq
and
\beq\label{eq:FloquetLb}\Lambda_F=\Lambda(\overline V, \overline R)=\overline{\Lambda(V,R)}.\eeq
\end{proposition}

\

\begin{proof}
Due to Corollary~\ref{co:nu1}, we have that
$$\exists\, C>0,\qquad\U_F(t,x)e^{\Lambda_Ft}=C\,\U(W(t);x)e^{\int_0^t\Lambda(W(s),R(s))\,ds}$$
and so, integrating in $x,$ we find
$$e^{\int_0^t\Lambda(W(s),R(s))\,ds-\Lambda_Ft}=C^{-1}\int_0^\infty\U_F(t,x)\,dx.$$
Then, by periodicity, we have
$$e^{\int_0^T\Lambda(W(s),R(s))\,ds-\Lambda_FT}=1,$$
which gives
$$\Lambda_F=\f1T\int_0^T\Lambda(W(s),R(s))\,ds=\f1T\int_0^T (\tau W(s)-\mu R(s))\,ds.$$
Due to ODE~\eqref{eq:ODEWF} we have
$$\int_0^TV-W=0,$$
and so
$$\Lambda_F=\f1T\int_0^T (\tau V(s)-\mu R(s))\,ds=\f1T\int_0^T\Lambda(V(s),R(s))\,ds.$$
\end{proof}

\

In the case $\nu=0$ and $\gamma=1,$ we cannot ensure the existence of Floquet eigenelements with our method.
Nevertheless, due to Corollary~\ref{co:nu0gamma1}, we can compare the eigenvalues of the reduced system~\eqref{eq:ODEclosed}, which is satified by $M_0[u]$ and $M_1[u].$

\medskip

\begin{proposition}
Assume that $\tau(x)=\tau,$ $\beta(x)=\beta x,$ $\mu(x)=\mu,$ and that $\kappa$ is symmetric.
Then we have the comparison
\beq\label{eq:FloPer}\overline{\Lambda(V,R)}\leq\Lambda_F\leq\Lambda(\overline V,\overline R).\eeq
\end{proposition}

\

\begin{proof}
Define $W$ as the periodic solution to
$$\dot W=\f{\Lambda(W,0)}{k}(V-W).$$
We know thanks to Corollary~\ref{co:nu0gamma1} that
$${\mathcal M}_0[W](t)=M_0[\U] e^{\int_0^t\Lambda(W(s),R(s))\,ds}$$
and
$${\mathcal M}_1[W](t)=M_1[\U] W^{k}(t) e^{\int_0^t\Lambda(W(s),R(s))\,ds}$$
solve System~\eqref{eq:ODEclosed}.
As a consequence
$$\Lambda_F=\f1T\int_0^T\Lambda(W(s),R(s))\,ds=\f1T\int_0^T(\sqrt{\beta\tau W(s)}-\mu R(s))\,ds.$$
Using the ODE satisfied by $W,$ we have
$$0=\int_0^T\f{\dot W}{W}=\f{\sqrt{\beta\tau}}{k}\int_0^T\f{V-W}{\sqrt{W}},$$
and we obtain that
$$\int_0^T\sqrt{W}=\int_0^T\f{V}{\sqrt{W}}.$$
Then the Caucy-Schwartz inequality gives
$$\int_0^T\sqrt{V}=\int_0^T\f{\sqrt{V}}{W^{\f14}}W^{\f14}\leq\sqrt{\int_0^T\f{V}{\sqrt{W}}}\sqrt{\int_0^T\sqrt{W}}=\int_0^T\sqrt{W},$$
and so
$$\f1T\int_0^T\Lambda(V,R)=\f1T\int_0^T\sqrt{\beta\tau V}-\mu R\leq\f1T\int_0^T\sqrt{\beta\tau W}-\mu R=\f1T\int_0^T\Lambda(W,R)=\Lambda_F.$$
To obtain the second inequality in \eqref{eq:FloPer} we write, using the ODE satisfied by $W,$
$$0=\int_0^T\f{\dot W}{\sqrt{W}}=\f{\sqrt{\beta\tau}}{k}\int_0^T V-W,$$
and so
$$\int_0^T V=\int_0^T W.$$
Thus we have, using the Jensen inequality,
$$\f1T\int_0^T\sqrt{W(s)}\,ds\leq\sqrt{\f1T\int_0^TW(s)\,ds}=\sqrt{\f1T\int_0^TV(s)\,ds},$$
and finally
$$\Lambda_F=\f1T\int_0^T(\sqrt{\beta\tau W(s)}-\mu R(s))\,ds\leq\sqrt{\f1T\int_0^T\beta\tau V(s)\,ds}-\f1T\int_0^T \mu R(s)\,ds=\Lambda(\overline V,\overline R).$$
\end{proof}

\

\section*{Conclusion and Perspectives}

We have introduced a new reduction method to investigate the long-time behavior of some nonlinear growth-fragmentation equations.
It allowed us to prove convergence and stability results when there is only one nonlinearity in the growth term,
and to prove the possible existence of nontrivial periodic solutions in cases when there are two competing nonlinearities.
The method is based on the study of exact solutions, whose existence requires powerlaw coefficients and a self-similar structure of the fragmentation kernel.
A further work would be to investigate more general growth-fragmentation equations for which no exact solution is available.

Consider for instance a generalization of Equation~\eqref{eq:nonlinear}, namely
\beq\label{eq:nonlinear_gen}\f{\p}{\p t} u(t,x) = - f\left(\int \psi(x) u(t,x)\,dx\right) \f{\p}{\p x} \bigl(\tau(x)\,u(t,x)\bigr) - \mu(x) u(t,x) + {\mathcal F} u(t,x),\eeq
where $\psi(x),\ \tau(x),\ \mu(x),$ and $\beta(x)$ are general positive functions, and $b(x,y)$ a general kernel without self-similar structure.
Is there convergence of any solution of Equation~\eqref{eq:nonlinear_gen} to a steady state?
This problem is a first step before tackling the same question for the original prion model~\eqref{eq:prion}, which is still an open problem for general coefficients.

Based on the study in this paper and on numerical simulations (see below),
we can conjecture that any bounded solution to Equation~\eqref{eq:nonlinear_gen} converges to a steady state (no oscillating solutions).
To ensure that any solution remains bounded, it should be sufficient to assume that
$$\limsup_{I\to\infty}\,\Lambda(f(I),1)<0,$$
which is a generalization of the second condition in Assumption~\eqref{as:f}.
To prove that, under this condition, all the solutions converge to a steady state, nonlinear entropy methods must be developed, which is a very challenging problem.

\

\underline{Numerical simulations.}
We choose coefficients which do not have the homogeneity of powerlaws and a fragmentation kernel which is not self-similar.
Then we numerically solve Equation~\eqref{eq:nonlinear_gen} and plot the quantity $\int_0^\infty\psi(x)u(t,x)\,dx$ along time for various initial distributions.
The convergence of this quantity to a constant (see Figure~\ref{fig:general}) indicates that the solution $u(t,x)$ converges to a steady state.
Indeed if the transport term is constant in time, we obtain a linear equation and the General Relative Entropy ensures the convergence of the solution to an eigenfunction.

\begin{figure}[h]
\centering\includegraphics[width=.58\textwidth]{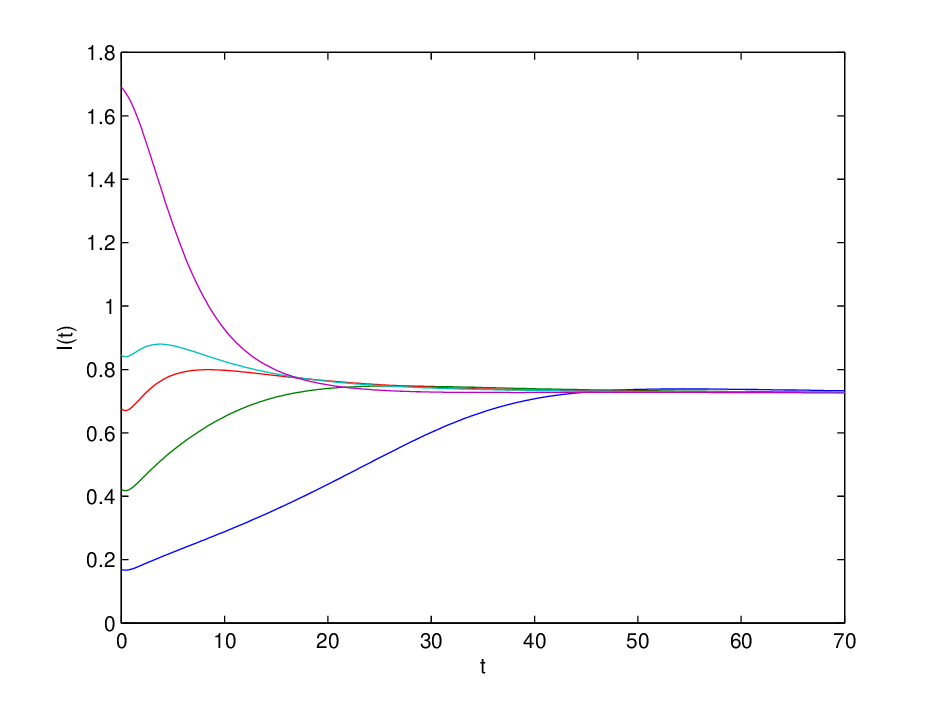}

\

\caption{We plot the evolution of $I(t):=\int_0^\infty\psi(x)u(t,x)\,dx$ for solutions $u(t,x)$ to Equation~\eqref{eq:nonlinear_gen} associated to various initial distributions.
The coefficients are $f(I)=1+\exp(-I^4),$ $\psi(x)=1+\sin(x),$ $\tau(x)=\f{x^2}{1+x^2},$ $\mu(x)=\ln(1+x),$ $\beta(x)=x\ln(1+x),$
and $b(x,y)=\beta(x)\f{1}{\sqrt{\pi}\text{erf}\left(\f x2\right)}\exp\left(-\left(x-\f y2\right)^2\right),$ where erf is the error function.
The fragmentation kernel chosen like this is equivalent to the homogeneous fragmentation $b_0(x,y)=\f{\beta(x)}{x}$ when $x\to0,$
and equivalent to the mitosis kernel $b_\infty(x,y)=\beta(x)\delta_{x=2y}$ when $x\to+\infty,$ and thus it is not self-similar.
We see that for any initial distribution, the $\psi th$-moment $I(t)$ converges to a constant.}\label{fig:general}

\end{figure}

\bigskip

{\bf Acknowledgements}

The author would like to thank Jean-Pierre Fran\c{c}oise for helpful discussions about limit cycles,
and Philippe Lauren\c cot for his corrections and suggestions.

\

%
%

\bibliographystyle{abbrv}
\bibliography{Prion}

\end{document}